\documentclass[pmlr,twocolumn,a4paper,10pt]{jmlr}
\usepackage{isipta2021}
\usepackage[british]{babel}
\usepackage{hyperref,url}
  \hypersetup{  % bookmarks
    bookmarksnumbered
  }
  \hypersetup{  % links
    breaklinks=false,
    pdfborderstyle={/S/U/W 0},  % zet op .5 ofzo voor onderlijnen
    citebordercolor=.235 .702 .443,
    urlbordercolor=.255 .412 .882,
    linkbordercolor=.804 .149 .19,
  }

\usepackage{amsmath,mathtools,amssymb,titlesec,scalerel,nicefrac,enumitem}
\usepackage{cmap}
\usepackage{microtype}
\usepackage[adobe-utopia]{mathdesign}
\usepackage[T1]{fontenc}
\usepackage{siunitx}
\usepackage{xfrac}
\usepackage{booktabs}
\usepackage{adjustbox}
\usepackage{mathcomp}
\usepackage{relsize}
\usepackage{xcolor}

\newenvironment{proofof}[1][\itseries]{\vspace*{10pt}\par\noindent{\bfseries\upshape Proof of #1\ }}{\jmlrQED}

\newcommand{\reals}{\mathbb{R}}
\newcommand{\extreals}{\overline{\mathbb{R}}}

\newcommand{\nnegextreals}{\overline{\mathbb{R}}_{\geq 0}}

\newcommand{\nnegreals}{\reals_{\geq 0}}
\newcommand{\nats}{\mathbb{N}}
\newcommand{\natz}{\mathbb{N}_{0}}

\newcommand{\indica}[1]{\mathbb{I}_{#1}}

\newcommand{\prob}{\mathrm{P}}

\newcommand{\prev}{\mathrm{E}}
\newcommand{\mprev}[1]{\prev_{#1}}

\newcommand{\dif}{\mathrm{d}}

\newcommand{\upprev}[1][]{\overline{\mathrm{E}}_{#1}}
\newcommand{\upprevvovk}[1]{\overline{\mathrm{E}}_{\mathrm{G},#1}}
\newcommand{\lowprev}{\underline{\mathrm{E}}}

\newcommand{\mupprev}[1]{\overline{\mathrm{E}}_{#1}}
\newcommand{\axupprev}[1]{\overline{\mathrm{E}}_{\mathrm{A}\,,#1}}
\newcommand{\mlowprev}[1]{\underline{\mathrm{E}}_{\,#1}}

\newcommand{\lupprev}[1]{\smash{\overline{\mathrm{Q}}_{#1}}}

\newcommand{\extlupprev}[1]{\overline{\mathrm{Q}}{}_{#1}^{\hspace{0.6pt}\raisebox{1pt}{\scalebox{0.6}{\ensuremath \uparrow}}}}

\newcommand{\sit}{x_{1:n}}

\newcommand{\martingale}{\mathscr{M}}

\newcommand{\setofsupmartb}{{\mathbb{M}}_\mathrm{b}}

\newcommand{\setofprob}[1][]{\mathbb{P}_{\,#1}}

\newcommand{\situations}{\mathscr{X}^\ast}
\newcommand{\statespace}{\mathscr{X}}

\newcommand{\samplespace}{\Omega}

\newcommand{\probtree}{\mathscr{P}}

\newcommand{\gengambles}{\mathscr{L}}
\newcommand{\genvariables}{\overline{\mathscr{L}}}

\newcommand{\fingambles}{\mathbb{F}}

\newcommand{\gambles}{\mathbb{V}}
\newcommand{\extvariables}{\overline{\mathbb{V}}}
\newcommand{\bextvariables}{\overline{\mathbb{V}}_\mathrm{b}}

\newcommand{\nnegextvariables}{\overline{\mathbb{V}}_{\geq 0}}
\newcommand{\nnegvariables}{\smash{\mathbb{V}{}_{\raisebox{5pt}{\scalebox{0.7}{\ensuremath{\geq 0}}}}^{\scalebox{0.7}{u}}}}

\newcommand{\posspace}{\mathscr{Y}}

\newcommand{\abs}[1]{\left\lvert#1\right\rvert}

\newtoggle{arxiv}
\toggletrue{arxiv}
\iftoggle{arxiv}{}{\jmlrworkshop{ISIPTA 2021}}

\title[Global Upper Expectations for Discrete-Time Stochastic Processes: In Practice, They Are All The Same!]{Global Upper Expectations for Discrete-Time Stochastic Processes: \\In Practice, They Are All The Same!
}

\author{% Authors are listed one per line and grouped per affiliation
  \Name{Natan T'Joens}\Email{natan.tjoens@ugent.be}\\
  \Name{Jasper De Bock}\Email{jasper.debock@ugent.be}\\ 
  \addr Foundations Lab for Imprecise Probabilities, ELIS, Ghent University, Belgium
}

\begin{document}
\maketitle

\begin{abstract}
We consider three different types of global uncertainty models for discrete-time stochastic processes: measure-theoretic upper expectations, game-theoretic upper expectations and axiomatic upper expectations.
The last two are known to be identical.
We show that they coincide with measure-theoretic upper expectations on two distinct domains:
monotone pointwise limits of finitary gambles, and bounded below Borel-measurable variables. We argue that these domains cover most practical inferences, and that therefore, in practice, it does not matter which model is used.
\end{abstract}
\begin{keywords}
upper expectation, imprecise probabilities, monotone convergence, probability measure, supermartingale, capacitability
\end{keywords}

\section{Introduction}

To describe the dynamics of a discrete-time stochastic process, one may choose between a number of different mathematical approaches.
There is of course the measure-theoretic option \cite{billingsley1995probability,shiryaev2016probabilityPartI,shiryaev2019probabilityPartII}---undoubtedly the most popular one---but one can also use martingales or game-theoretic principles to do so \cite{Shafer:2005wx,Vovk2019finance,williams1991probability}. 
Each of these approaches has its own unique strengths and flaws, and each of them---rightly or not---has attracted a dedicated group of followers.
Our aim here is not to argue for the use of one or the other though, but rather to study the mathematical relation between the (global) uncertainty models that arise from these approaches in a general, imprecise-probabilistic context.
As we will see, they turn out to be surprisingly similar.
% to establish a mathematical relation between the global uncertainty models that arise from these approaches.

All the global---imprecise---uncertainty models that we will consider take the form of an upper (or lower) expectation \cite{troffaes2014,Walley:1991vk}; 
a non-linear operator that can---but need not---be interpreted as a tight upper bound on a set of expectations.
% They are part of a larger collection of so-called imprecise probability models \cite{Walley:1991vk,troffaes2014,Augustin:2014di} which, in general, generalise traditional models to allow for partially specified para
They are called global because they model beliefs about the entire, uncertain path taken by the process.
In that sense, they differ from---and are more general than---local uncertainty models, which only give information about how the process is likely to evolve from one time instant to the next.
Such local models form the parameters of a stochastic process, whereas the global uncertainty model that follows from it---in our case, a global upper expectation---extends the information incorporated in these local models.  
It is the particular way in which this extension is done that distinguishes one type of global model from the other.

We consider three global models.
% ; so, three methods for extending local models.
The first is a probabilistic model that is defined as an upper envelope over a set of measure-theoretic global expectations \cite{8535240,TJOENS202130}.
The second is based on game-theoretic principles, and defined as an infimum over hedging prices; see Refs.~\cite{Shafer:2005wx,Vovk2019finance}.
The last is an abstract axiomatic model, whose defining axioms we have motivated in an earlier paper \cite{TJOENS202130} on the basis of both a probabilistic and a behavioural interpretation.
We have already shown that the second and third of these three global upper expectations are identical \cite{TJOENS202130}.
In this paper, we relate the first---measure-theoretic---one to this common axiomatic/game-theoretic upper expectation.

Our contribution consists in showing that they are equal on two different domains:
variables that are monotone (upward or downward) limits of finitary gambles---bounded variables that only depend on the process' state at a finite number of time instances---and bounded below Borel-measurable variables.
Upper expectations on these two types of domains cover the vast majority of inferences encountered in practice;
upper and lower\footnote{Lower expectations can be derived from upper expectations using conjugacy; see Section~\ref{Sect: Preliminaries} and Corollary~\ref{corollary: equivalence}.} expected hitting times, for instance, fall under the first category \cite{8627473}; upper and lower expected time averages under the second~\cite{TJOENS2021181}.
% Together, these two types of domains cover the vast majority of inferences encountered in practice; 
% hitting times, for instance, fall under the first category \cite{8627473}; time averages under the second~\cite{TJoens_IPMU2020_weak_ergodicity}.
Hence the title of this paper.
That the three considered global upper expectations are equal on such a large domain is relevant in a number of ways.
First of all, it leaves no room for discussion when it comes to choosing a global model; it simply does not matter since all of them are equal.
Philosophically speaking, it is interesting that, whatever the interpretational point of view and associated system of logical reasoning is, we always end up with exactly the same object.
Finally, and maybe most importantly, such a relation provides us with a large number of additional mathematical properties for the models at hand;
properties that were previously only known to hold for one or two of these models, suddenly hold for all three of them.
We refer to Refs.~\cite{DeBock2021Ergodic,8627473,8535240} for an illustration of how properties acquired in this way have already led to important consequences.

% Before we start, we would like to mention that some of our proofs---more specifically, that of Proposition~\ref{Prop: decreasing continuity measure-theoretic} and that of Theorem~\ref{theorem: equivalence for measurable functions}---rely on ideas presented in recent work by Shafer and Vovk \cite[Chapter~9]{Vovk2019finance}.
% In fact, at first glance, the results in \cite[Chapter~9]{Vovk2019finance} may appear extremely similar to ours.
% We explain in Section~\ref{Sect: Relation with Shafer and Vovk} that this is not the case, by pointing out some important differences.
 % to be almost equivalent to the ones that we will establish here.
% Make no mistake though;
% as we will point out in Section~\ref{Sect: Relation with Shafer and Vovk}, there are some key differences, therefore making our results more than just a trivial extension or modification of the ones in \cite[Chapter~9]{Vovk2019finance}. 

\iftoggle{arxiv}
{This paper is an extended version of a contribution that is submitted for possible publication in the Proceedings of ISIPTA 2021.
Compared to the submitted version, this extended version additionally includes an appendix containing proofs for the results in the main text.}
{To adhere to the page limit, we have relegated most of our proofs to the appendix of an extended online version of this paper \cite{TJoens2021Equivalence}.}

\section{Local Uncertainty Models}

A \emph{discrete-time stochastic process} is an infinite sequence \(X_1, X_2, ..., X_k , ...\) of uncertain states, where the state \(X_k\) at each discrete time point \(k \in \nats\) takes values in a fixed non-empty set \(\statespace{}\), called the \emph{state space}.
We will assume that this state space $\statespace{}$ is \emph{finite}.
% ---an assumption that is crucial for our further results to hold.
Typically, when modelling the dynamics of a stochastic process, one starts off on a local level, by specifying how the process' state $X_k$ is (likely) to evolve from one time instant to the next.
In particular, we do this by attaching a so-called local uncertainty model to each possible \emph{situation};
a finite---possibly empty---sequence $x_{1:k} \coloneqq x_{1} x_{2}\cdots x_{k}$ of state values that represents a possible history $X_1 = x_1 , \cdots, X_k = x_k$ up until some time point $k\in\natz{}$, with $\natz\coloneqq\nats{}\cup\{0\}$.
The local model associated with the situation $x_{1:k}$ then models beliefs about the value of the next state $X_{k+1}$, conditional on the history represented by $x_{1:k}$.
We let $\situations{}\coloneqq \cup_{i\in\natz{}}\statespace{}^i$ be the set of all situations and we denote the \emph{initial (empty) situation} by $\square \coloneqq x_{1:0}=\statespace{}^0$.

% This information is mathematically represented by attaching a local uncertainty model to each \emph{situation}:
% a finite---possibly empty---sequence $x_{1:k} \coloneqq x_{1} x_{2}\cdots x_{k} \in\statespace{}^k$ of state values.
% Any situation $x_{1:k}\in\situations{}\coloneqq \cup_{i\in\natz{}}\statespace{}^i$ can be interpreted as denoting a possible history $X_1 = x_1 , X_2 = x_2 , \cdots, X_k = x_k$ of the process, and the local uncertainty model associated with this situation then expresses beliefs
% This information is mathematically represented by what we call the `local uncertainty models'.
% They are the parameters that characterise a stochastic process, and the starting point for any global uncertainty model.

Among the most popular types of local uncertainty models are (probability) mass functions $p$ on $\statespace{}$; for any situation $x_{1:k}\in\situations{}$, the mass function $p(\cdot \vert x_{1:k})$ then provides, for each $x_{k+1}\in\statespace{}$, the probability $p(x_{k+1} \vert x_{1:k})$ that the value of the state $X_{k+1}$ will be equal to $x_{k+1}$.
Such a family of probability mass functions is represented by a single function $p\colon s\in\situations{}\mapsto p(\cdot\vert s)$, which we call a \emph{precise probability tree}.\footnote{The reason why we call it a `tree' is because it is a map on $\situations{}$, which can naturally be visualised in terms of infinite (event) trees \cite[Figure~1]{DECOOMAN201618}.}
% The value of the initial state $X_1$ is additionally described by a mass function $p(\cdot\vert \Box)$, where $\Box = x_{1:0}\in\statespace{}^0$ denotes the empty sequence.
What is equivalent, but less of a popular habit, is to attach to each possible situation $x_{1:k}\in\situations{}$ an expectation $\prev{}_{x_{1:k}}$ on the set $\gengambles{}(\statespace{})$ of all real-valued functions $f$ on $\statespace{}$.  
These expectations $\prev{}_{x_{1:k}}$ may then be interpreted in a measure-theoretic sense, as coming from an underlying family of mass functions $p(\cdot\vert x_{1:k})$, but they can also be interpreted in a direct behavioural way as a subject's fair prices, as De Finetti does~\cite{DeFinetti:oT_PWtAE}.
% In that case, $\prev{}_{x_{1:k}}(f)$ represents a subject's fair (selling and buying) price for the gamble $f\in\gengambles{}(\statespace{})$ whose payoff is equal to $f(x_{k+1})$ if the state $X_{k+1}$ at the next time instant takes the value $x_{k+1}\in\statespace{}$. 

Unfortunately, irrespective of one's preference between mass functions and linear expectations, both of them are rather inadequate when modelling situations where data is scarce, or when modelling the beliefs of a conservative (risk-averse) subject.
% who may not want to buy gambles at the same price as he wants to sell them.
In such situations, one can reach for so-called `imprecise' probability models \cite{troffaes2014,Walley:1991vk,Augustin:2014di}.
These come in many different shapes and forms (e.g. sets of desirable gambles, belief functions, credal sets,...), but, for our purpose of modelling the local dynamics of a process, we will only consider two specific---yet wide-spread---ones; credal sets and coherent upper (and lower) expectations.

The first, \emph{credal sets}, are closed (under the topology of pointwise convergence) convex sets of probability mass functions; see e.g. \cite[Section 9.2]{Augustin:2014di}.
If we attach to each situation $s\in\situations{}$ a credal set $\probtree_{s}$ on $\statespace{}$, then we obtain a so-called \emph{imprecise probability tree} $\probtree_{\text{\relscale{0.8}$\bullet$}}\colon s\in\situations{}\mapsto \probtree_{s}$, which we will often simply denote by $\probtree{}$.
% that maps each situation $s\in\situations{}$ to its corresponding credal set $\probtree_{s}$
For any $s\in\situations{}$, the associated credal set $\probtree_{s}$ may then be interpreted as a set that contains all local mass functions $p(\cdot\vert s)$ that are deemed `possible'.
% Given such a credal set $\probtree_{s}$ for each $s\in\situations{}$, we gather all of them in a single function $\probtree{}\colon s\in\situations{}\mapsto \probtree_{s}$, called an \emph{imprecise probability tree}.
% \footnote{The reason why we call it a `tree' is because it is a map on $\situations{}$, which can naturally be visualised in terms of infinite (event) trees \cite[Figure~1]{DECOOMAN201618}.}
Such an imprecise probability tree $\probtree$ parametrises the stochastic process as a whole, and clearly does so in a more general manner than the precise methods mentioned earlier; precise probability trees correspond to the special case where, for each $s\in\situations{}$, $\probtree_{s}$ consists of a single mass function $p(\cdot\vert s)$.
% Such a special tree will henceforth be called a \emph{precise probability tree} $p\colon s\in\situations{}\mapsto p(\cdot\vert s)$.
We say that a precise probability tree $p$ is \emph{compatible} with an imprecise probability tree $\probtree$, and write $p\sim\probtree$, if $p(\cdot\vert s)\in\probtree_{s}$ for all $s\in\situations{}$.

Another---yet equivalent---approach consists in specifying a \emph{local coherent upper (or lower) expectation} $\lupprev{s}$ for each $s\in\situations{}$ \cite{Walley:1991vk}: a real-valued function on $\gengambles{}(\statespace{})$ that satisfies, for all $f,g\in\gengambles{}(\statespace{})$ and $\lambda\in\nnegreals{}$,
\begin{enumerate}[leftmargin=*,ref={\upshape{}C\arabic*},label={\upshape{}C\arabic*}.,itemsep=3pt, series=sepcoherence ]
\item \label{local coherence: upper bound} 
$\lupprev{s}(f) \leq \sup f$ \hfill \text{[upper bounds]};  
\item \label{local coherence: sub-additivity} 
$\lupprev{s}(f+g) \leq \lupprev{s}(f) + \lupprev{s}(g)$ \hfill \text{[sub-additivity]};
\item \label{local coherence: nneg homogeneity}
$\lupprev{s}(\lambda f) = \lambda \lupprev{s}(f)$ \hfill \text{[non-negative homogeneity]}.
\end{enumerate}
Any such family $(\lupprev{s})_{s\in\situations{}}$ of local coherent upper expectations will be gathered in a single \emph{upper expectation tree} \/ $\smash{\overline{\mathrm{Q}}}_{\text{\relscale{0.8}$\bullet$}} \colon s\in\situations{}\mapsto\lupprev{s}$, which we will also simply denote by $\smash{\overline{\mathrm{Q}}}$.
For any $x_{1:k}\in\situations{}$, the upper expectation $\lupprev{x_{1:k}}$ can be interpreted as representing a subject's minimum selling prices---a generalisation of De Finetti's fair price interpretation for linear expectations.\footnote{Traditionally, the behavioural interpretation of coherent upper expectations says that they represent infimum selling prices, rather than minimum selling prices; see Ref. \cite{Walley:1991vk}.
We opt for minimum selling prices here because they fit more naturally with the supermartingales that we will introduce further on.}
More concretely, this interpretation says that, given a situation $x_{1:k}\in\situations{}$ and any $f\in\gengambles{}(\statespace{})$, our subject is willing to sell the uncertain---possibly negative---payoff $f(X_{k+1})$ for any price $\alpha \geq \lupprev{x_{1:k}}(f)$.
Axioms~\ref{local coherence: upper bound}--\ref{local coherence: nneg homogeneity} can then be seen as rationality criteria.
% In other words, he is willing to commit to the combined gamble $\alpha - f$.
% Note, however, that this interpretation by no means implies that our subject is also willing to buy a gamble $f$ for any price $\alpha < \lupprev{x_{1:k}}(f)$, that is, $\lupprev{x_{1:k}}(f)$ may be higher than our subject's supremum buying price for $f$.
% It is this flexibility to cope with situations where there is a gap between infimum selling prices and supremum buying prices, what makes coherent upper (and lower) expectations so attractive when comparing them to linear expectations---which are simultaneously interpreted as infimum selling prices and supremum buying prices.
% Of course, this comes at a price: the Axioms~\ref{local coherence: sub-additivity} and~\ref{local coherence: nneg homogeneity} which enable this flexible character are weaker than the usual linearity properties of linear expectations.
We refer to Walley's work \cite{Walley:1991vk} for a more detailed motivation and justification for coherent upper (and lower) expectations.

Mathematically speaking, it does not matter whether we use imprecise probability trees or upper expectation trees to characterise a stochastic process, because credal sets and coherent upper expectations---and therefore imprecise probability trees and upper expectation trees---are in a one-to-one relation with each other.
In particular, with any imprecise probability tree $\probtree$, we can associate an upper expectation tree $\smash{\overline{\mathrm{Q}}}_{\text{\relscale{0.8}$\bullet$},\probtree{}}$ that maps each situation $s\in\situations{}$ to the upper envelope $\lupprev{s,\probtree{}}$ of the linear expectations corresponding to $\probtree_{s}$:
\begin{equation*}
\lupprev{s,\probtree{}}(f)\coloneqq\sup\Big\{\sum_{x\in\statespace{}}f(x)p(x\vert s) \colon p(\cdot\vert s)\in\probtree_{s}\Big\},
\end{equation*}
for all $f\in\gengambles{}(\statespace{})$.
That each $\lupprev{s,\probtree{}}$ is indeed a local coherent upper expectation follows from \cite[Theorem 3.6.1]{Walley:1991vk}.
Conversely, with any upper expectation tree $\smash{\overline{\mathrm{Q}}}$, we can associate an imprecise probability tree $\probtree_{\text{\relscale{0.8}$\bullet$},\smash{\overline{\mathrm{Q}}}}$;
for any $s\in\situations{}$, its local credal set $\probtree_{s,\smash{\overline{\mathrm{Q}}}}$ is the closed convex set of all mass functions $p(\cdot\vert s)$ that are dominated by $\lupprev{s}$, in the sense that
\begin{equation*}
\sum_{x\in\statespace{}}f(x)p(x\vert s) \leq \lupprev{s}(f) \text{ for all } f\in\gengambles{}(\statespace{}).
\end{equation*}
It follows once more from \cite[Theorem 3.6.1]{Walley:1991vk} that this correspondence between upper expectation trees and imprecise probability trees is one-to-one; that is, the map $\probtree\mapsto\smash{\overline{\mathrm{Q}}}_{\text{\relscale{0.8}$\bullet$},\probtree{}}$ is bijective and $\smash{\overline{\mathrm{Q}}}\mapsto\probtree_{\text{\relscale{0.8}$\bullet$},\smash{\overline{\mathrm{Q}}}}$ is its inverse.
We say that an imprecise probability tree $\probtree$ and an upper expectation tree $\smash{\overline{\mathrm{Q}}}$ \emph{agree} if they are related through these mappings.
% Furthermore, we will often simply use $\probtree{}$ to denote a generic imprecise probability tree $\probtree$, and similarly for $\smash{\overline{\mathrm{Q}}}$.

An important consequence of the one-to-one relation described above is that imprecise probability trees and upper expectation trees can borrow each others interpretation;
%  also allows local credal sets $\probtree_{s}$ and local upper expectations $\lupprev{s}$ to be simply regarded as secondary objects, derived from local upper expectations or local credal sets, respectively.
% As a result, they can borrow each others interpretation;
local credal sets can be interpreted as representing a subject's infimum selling prices, whereas local upper expectations can be interpreted as upper envelopes of the linear expectations associated with an underlying local credal set.
% and, more specifically, upper (bounds on) probabilities can be interpreted in a behavioural way, as infimum selling prices for $\{0,1\}$-valued gambles, whereas local upper expectations can be interpreted in a probabilistic way, as upper envelopes over linear expectations coming from an underlying local credal set.

% ***Though these two types of local descriptions are mathematically equivalent, each of them will naturally lead to its own approach when it comes to extending towards global models
% each of these two methods fundamentally differ in their 
%  two methods described above lead to a different extension on a global level; imprecise probability trees are naturally associated with a measure-theoretic approach, whereas upper expectation trees are more naturally extended using game-theoretic principles.
% We prefer expressing our global models in terms of upper expectations because they have a universal character, interpretationally wise, because they hold the mid between sets of probability measures and sets of desirable gambles.
% ***

\section{Three Types of Global Models}

Imprecise probability trees and upper expectation trees describe the dynamics of a stochastic process on a local level---how it changes from one time instant to the next---but they do not tell us anything, at least not directly, about more global features that relate to multiple time instances at once;
e.g. the time it takes until the process is in a given state $x\in\statespace{}$.
% They only form the mere starting point for a further mathematical analysis.
We therefore face the following question.
How do we turn the local information captured by any of these trees into global information about the process as a whole?
Three possible solutions are described in the current section, but we start by introducing some necessary terminology and notation.

\subsection{Preliminaries}\label{Sect: Preliminaries} 
% Regarding the Global Dynamics of Stochastic Processes}

A \emph{path} $\omega = x_1 x_2 x_3 \cdots$ is an infinite sequence of state values and represents a possible evolution of the process. 
The \emph{sample space}\, $\samplespace{} \coloneqq \statespace{}^\nats$ denotes the set of all paths.
For any $\omega=x_1 x_2 x_3 \cdots\in\samplespace$, we let $\omega^k\coloneqq x_{1:k} \in\statespace{}^k$ be the finite sequence that consists of the initial $k$ state values, and we let $\omega_k\coloneqq x_k \in\statespace$ be the $k$-th state value.
% Any finite sequence $x_{1:k}\in\situations{}$---which we have so far been referring to as a possible history---will from now on be called a \emph{situation}.
An \emph{event} $A \subseteq \samplespace$ is a set of paths and, in particular, for any situation $x_{1:k}\in\situations{}$, the \emph{cylinder event} $\Gamma(x_{1:k}) \coloneqq \{\omega\in\samplespace\colon \omega^k = x_{1:k}\}$ is the set of all paths that go through the situation $x_{1:k}$.

We let $\extreals\coloneqq\reals\cup\{+\infty, -\infty\}$ be the extended real numbers, $\nnegextreals$ be the subset of non-negative ones, and $\nnegreals{}$ be those that are moreover real.
% The set of positive real numbers is denoted by $\posreals$, 
We extend the total order relation $<$ on $\reals{}$ to $\extreals$ by positing that $-\infty<c<+\infty$ for all $c\in\reals{}$ and endow $\extreals{}$ with the associated order topology.

Any extended real-valued function $f\colon\posspace{}\to\extreals{}$ on some non-empty set $\posspace{}$ will be called a \emph{variable}.
Any bounded variable---that is, a variable $f$ for which there is a $B\in\nnegreals{}$ such that $-B\leq f(y)\leq B$ for all $y\in\posspace{}$---will be called a \emph{gamble}.\footnote{This choice of terminology is due to Walley \cite{Walley:1991vk}. However, for us, the mathematical object of a gamble is not necessarily bound to the interpretation as an uncertain payoff.}
The set of all variables will be denoted by $\genvariables{}(\posspace{})$ and the set of all gambles by $\gengambles{}(\posspace{})$.
Note that this definition is in accordance with our earlier use of $\gengambles{}(\statespace{})$, where it denoted the real-valued functions on $\statespace{}$---which are automatically bounded because $\statespace{}$ is finite.
The elements of $\genvariables{}(\statespace{})$ and $\gengambles{}(\statespace{})$ are called \emph{local} variables and gambles, respectively.
On the other hand, the variables in $\extvariables{}\coloneqq\genvariables{}(\samplespace{})$ and $\gambles{}\coloneqq\gengambles{}(\samplespace{})$ are called \emph{global} variables and gambles, respectively; they may depend on the entire path $\omega\in\samplespace{}$ taken by the process.
Variables that only depend on the process' state at a finite number of time instances are called \emph{finitary}; for such a finitary variable $f\in\extvariables{}$, there is an $n\in\nats{}$ and some $g\in\genvariables{}(\statespace{}^n)$ such that $f(\omega) = g(\omega^n)$ for all $\omega\in\samplespace{}$.
We often make this explicit by writing $f=g(X_{1:n})$, where $g(X_{1:n})\coloneqq g \circ X_{1:n}$ and where $X_{1:n}$ is the projection of $\omega\in\samplespace{}$ on its first $n$ state values $\omega^n$.
Sometimes, we also allow ourselves a slight abuse of notation by writing $f(x_{1:n})$ to denote the constant value of $f(\omega)=g(x_{1:n})$ on all paths $\omega\in\samplespace{}$ such that $\omega^n = x_{1:n}$.
We collect all finitary gambles in the set $\fingambles{}$.
A special type of global gamble is the \emph{indicator} $\indica{A}$ of an event $A$, which assumes the value $1$ on $A$ and $0$ elsewhere.
For any $s \in \situations{}$, the indicator $\indica{s} \coloneqq \indica{\Gamma(s)}$ of the cylinder event $\Gamma(s)$ is clearly a finitary gamble.
% Finally, we call any $f \in \extvariables{}$ \emph{finitary} if it is $n$-measurable for some $n \in \natz{}$.

A \emph{global upper expectation}, finally, is a map $\upprev{} \colon \extvariables{} \times \situations{} \to \extreals{}$; it maps global variables $f\in\extvariables{}$ and situations $s\in\situations{}$ to a corresponding (conditional) upper expectation $\upprev{}(f \vert s)$.
As we will see, such maps can play the role of a global uncertainty model, in the sense that they can represent beliefs or knowledge about the path $\omega$ taken by the process, or about the value attained by a global variable $f$.
% We proceed to present three different types of them.
Apart from global upper expectations, one can also consider \emph{global lower expectations} $\lowprev{}\colon\extvariables{}\times\situations{}\to\extreals{}$; 
for each of the models that we will consider, these are conjugate to the corresponding global upper expectation $\upprev{}$, in the sense that $\lowprev{}(f\vert s) = -\upprev{}(-f\vert s)$ for all $f\in\extvariables{}$ and $s\in\situations{}$.
It therefore suffices to focus on only one of them;
our theoretical developments focus on $\upprev{}$, leaving the implications for $\lowprev{}$ for Section~\ref{Sect: Relation with Shafer and Vovk}.
% This relation---called the conjugacy relation---between (global) upper and lower expectations ensures that (global) upper and lower expectations are equivalent; we choose to work with upper expectations exclusively.
% Section~\ref{Sect: Relation with Shafer and Vovk}, however, is an exception, and will illustrate how results about upper expectations can be combined to obtain a result that applies to both upper and lower expectations simultaneously.

% we are given an upper expectation tree $\smash{\overline{\mathrm{Q}}}$, and that we want to deduce from it a model that (directly) gives information about the path $\omega$ taken by the process or about the value attained by a global variable $f$.
% In specific, we want to deduce from it a global model in the form of a \emph{global upper expectation}, which is an extended real-valued map $\upprev{} \colon \extvariables{} \times \situations{} \to \extreals{}$.

\subsection{Measure-Theoretic Global Upper Expectations}\label{Sect: Measure-Theoretic Global Upper Expectations}

We start by presenting a traditional measure-theoretic approach, where global upper expectations are defined as upper envelopes of sets of (linear) expectations, and where each of these (linear) expectations on its turn is derived from a different probability measure on $\samplespace{}$. 
% One possible way to justify the axioms \ref{P compatibility}--\ref{P continuity} for the model \(\axupprev{\smash{\overline{\mathrm{Q}}}}\), is to interprete \(\axupprev{\smash{\overline{\mathrm{Q}}}}\) as an upper envelope of linear expectations.
% Our notion of a linear expectation, and the notion of an underlying probability, is in this case largely informal, though. 
% In this section, we present a very strict and formal measure-theoretic approach, where each (global) linear expectation is derived by Lebesgue integration from an underlying probability measure on $\samplespace{}$.

% The axiomatic model \(\axupprev{\smash{\overline{\mathrm{Q}}}}\) can be motivated from a probabilistic point of view, in the sense that its defining axioms are properties that seem justified for an upper envelope of linear expectations.
% However, we did not give a formal definition or a clear interpretation for a linear expectation though; we simply relied on the common sense intuition about linear expectations.
% Yet, as will be presented next, we can also adopt a very strict and formal measure-theoretic approach, where each (global) linear expectation is derived by Lebesgue integration from an underlying probability measure on $\samplespace{}$. 
  
Consider an imprecise probability tree $\probtree{}$ and let $p\sim\probtree{}$ be any precise probability tree that is compatible with $\probtree{}$.
With each $x_{1:k}\in\situations{}$, we associate a probability measure $\prob_{p}(\cdot\vert x_{1:k})$ on the $\sigma$-algebra $\mathscr{F}$ generated by all cylinder events as follows.
First, for any \(\ell \in \natz{}\) and any \(C \subseteq \statespace{}^\ell\), let
\begin{align}\label{Eq: precise probability on algebra}
&\prob_{p}(C \vert x_{1:k}) 
\coloneqq
\prob_{p}(\cup_{z_{1:\ell}\in C} \Gamma(z_{1:\ell}) \vert x_{1:k}) 
\coloneqq \sum_{z_{1:\ell} \in C} \prob_{p}(z_{1:\ell} \vert x_{1:k}), \nonumber \\ &\text{where }
\prob_{p}(z_{1:\ell} \vert x_{1:k}) \coloneqq \\
&\hspace*{35pt}
\begin{aligned}
\begin{cases}
\prod_{i=k}^{\ell-1} p( z_{i+1} \vert z_{1:i}) &\text{ if }  k < \ell 
% \\ &\hspace{0.15cm}
\text{ and } z_{1:k} = x_{1:k} \\
1 &\text{ if } k \geq \ell 
% \\ &\hspace{0.15cm}
\text{ and } z_{1:\ell} = x_{1:\ell}  \\
0 &\text{ otherwise. } \nonumber
\end{cases}
\end{aligned}
\end{align}
% , and where we did not make any particular distinction between the notations \(z_{1:\ell}\) and \(\{z_{1:\ell}\}\).
It is then easy to see that, on the algebra generated by the cylinder events, $\prob_{p}(\cdot \vert x_{1:k})$ forms a finitely additive probability \cite[Chapter 3]{8535240}.
Hence, by \cite[Theorem 2.3]{billingsley1995probability}, it is also a countably additive probability---that is, a probability measure---on this algebra and so, by Carathéodory's extension theorem \cite[Theorem 1.7]{williams1991probability}, $\prob_{p}(\cdot \vert x_{1:k})$ can be uniquely extended to a probability measure on \(\mathscr{F}\).

In accordance with standard practices, we then associate with every probability measure $\prob_{p}(\cdot \vert s)$ an expectation $\mprev{p}(\cdot\vert s)$ using Lebesgue integration.
That is, we let $\smash{\mprev{p}(f \vert s)\coloneqq \int_{\samplespace{}} f \dif{\prob_{p}(\cdot \vert s)}}$ for all $f\in\extvariables{}$ for which $\smash{\int_{\samplespace{}} f \dif{\prob_{p}(\cdot \vert s)}}$ exists, which is guaranteed if $f$ is $\mathscr{F}$-measurable and bounded below (or bounded above).
% and if \(\min\{\int_{\Omega} f^{+} \dif \mathrm{P}_\mathrm{u},\int_{\Omega} f^{-} \dif \mathrm{P}_\mathrm{u}\} < +\infty\), with \(f^{+} \coloneqq f^{\vee 0}\) and \(f^{-}\coloneqq -f^{\wedge 0}\).
For general $f\in\extvariables{}$, we adopt an upper integral $\mupprev{p}(f \vert s)$ defined by \vspace*{-2pt}
\begin{equation}\label{Eq: upper integral}
\mupprev{p}(f \vert s) 
\coloneqq \inf \Bigl\{ \mprev{p}(g \vert s) \colon g\in\overline{\mathbb{V}}_{\sigma,\mathrm{b}} \text{ and } g \geq f \Bigr\},\vspace*{-2pt}
\end{equation}
where $\overline{\mathbb{V}}_{\sigma,\mathrm{b}}$ is the set of all bounded below $\mathscr{F}$-measurable variables in $\extvariables{}$.
It follows from \cite[Proposition~12]{TJOENS202130} that $\mupprev{p}(\cdot\vert s)$ coincides with $\mprev{p}(\cdot\vert s)$ on the entire domain where $\mprev{p}(\cdot\vert s)$ is well-defined---that is, where the Lebesgue integral with respect to $\prob_{p}(\cdot\vert s)$ exists---and hence, that $\mupprev{p}(\cdot\vert s)$ is an extension of $\mprev{p}(\cdot\vert s)$.

Finally, the global upper expectation $\mupprev{\probtree{}}$ corresponding to the imprecise probability tree $\probtree{}$ is defined as the upper envelope of the upper integrals $\smash{\mupprev{p}}$ corresponding to each of the precise trees $\smash{p\sim\probtree{}}$. 
That is, for each $f\in\extvariables{}$ and $s\in\situations{}$,\vspace*{-2pt}
\begin{align*}
\mupprev{\probtree{}}(f \vert s)\coloneqq \sup\big\{\mupprev{p}(f \vert s) \colon p\sim\probtree{}\big\}.\vspace*{-2pt}
\end{align*}
This definition is in line with the sensitivity analysis interpretation for imprecise probability models \cite[Section~1.1.5]{Walley:1991vk}, which regards them as resulting from a lack of knowledge about a single ideal precise model.

The approach set out above should look familiar to anyone with a measure-theoretic background, and we therefore omit an in-depth conceptual discussion; we instead refer to \cite[Section~9]{TJOENS202130} for more details.
One aspect, however, that we feel is worth pointing out is the difference between our way of conditioning and what is usually done in measure-theory.
Usually, conditional expectations (and probabilities) are derived from a single unconditional probability measure through the Radon-Nikodym derivative \cite[Section 2.7.2]{shiryaev2016probabilityPartI}.
We, on the other hand, associate with each situation $s\in\situations{}$ a separate---in the traditional sense, unconditional---probability measure $\prob_{p}(\cdot\vert s)$ and use this probability measure $\prob_{p}(\cdot\vert s)$ to define the expectation $\mprev{p}(\cdot\vert s)$.
The reason why we do so is because, unlike the traditional approach, it allows us to condition---in a meaningful way---on (cylinder) events with probability zero;
% this method preserves all the `information' captured by the precise probability tree $p$; 
again, we refer to \cite[Section~9]{TJOENS202130} for more details.

\subsection{Game-Theoretic Global Upper Expectations}

The second global model that we will consider is the game-theoretic upper expectation introduced and, for the most part, developed by Shafer and Vovk \cite{Shafer:2005wx,Vovk2019finance}.
% This upper expectation is more or less in line with the behavioural interpretation presented in Section~\ref{Sect: Axiomatic Global Upper Expectations}, 
This operator is defined in terms of infimum hedging prices; 
starting capitals that allow a gambler to cover---or hedge---the costs or gains of a given global gamble.
% can be regarded worth more than the corresponding global gamble, becatherefore acting as possible selling prices.
These hedging prices---and hence, these game-theoretic upper expectations---are determined using the notion of a supermartingale; a function that describes the possible evolution of a gambler's capital as he gambles in a way that is in accordance with the local models $\lupprev{s}$.

Formally, for any upper expectation tree $\smash{\overline{\mathrm{Q}}}$, a \emph{supermartingale} $\martingale{}$ is a real-valued function on $\situations{}$ that satisfies $\lupprev{s}(\martingale{}(s \cdot)) \leq \martingale{}(s)$ for all $s\in\situations{}$, where $\martingale{}(s \cdot)\in\gengambles{}(\statespace{})$ denotes the local gamble that takes the value $\martingale{}(s x)$ in $x\in\statespace{}$.
How can such a supermartingale be interpreted in the way described above?
Consider any situation $x_{1:k}\in\situations{}$ and a gambler---called `Skeptic' in Shafer and Vovk's framework---whose current capital equals $\martingale{}(x_{1:k})$.
Then, recalling our interpretation for the local model $\lupprev{x_{1:k}}$ as representing a subject's minimum selling prices, the condition that $\lupprev{x_{1:k}}(\martingale{}(x_{1:k} \cdot)) \leq \martingale{}(x_{1:k})$ implies that Skeptic can use his capital $\martingale{}(x_{1:k})$ to buy the uncertain reward $\martingale{}(x_{1:k}X_{k+1})$ 
% whose payoff depends on the next state $X_{k+1}$ 
from this subject---called `Forecaster' in Shafer and Vovk's framework.
% Forecaster reckons this to be a desirable (or acceptable) transaction---because of how he specified the local model $\lupprev{x_{1:k}}$---and so, 
If Skeptic chooses to commit to such a transaction, he is actually gambling \emph{against} Forecaster, which explains why these players are called Skeptic and Forecaster.
So we see that a supermartingale describes the evolution of Skeptic's capital if he chooses, in each situation, to buy a gamble that Forecaster is willing to sell.
% , then, according to the subject that provided the local upper expectations $\lupprev{s}$---called `Forecaster' in Shafer and Vovk's framework---this transaction is likely to decrease Skeptic's---and increase Forecaster's---capital, hence the reason Forecaster is offering such a transaction to Skeptic.
% So, if Skeptic chooses to commit to such a transaction, he is actually gambling \emph{against} Forecaster, which explains why these players are called Skeptic and Forecaster.
% In this way, we see that a supermartingale describes the evolution of Skeptic's capital if he chooses, in each situation, such a gamble.

A \emph{hedging price} $\alpha\in\reals{}$ for any $f\in\gambles{}$ is now a real number for which there is a bounded below supermartingale $\martingale{}$ that starts in $\martingale{}(\Box)=\alpha$ and such that $\liminf \martingale{}(\omega)\coloneqq\liminf_{k\to+\infty}\martingale{}(\omega^k) \geq f(\omega)$ for all $\omega\in\samplespace{}$.
A hedging price $\alpha$ for $f$ is therefore worth more to Skeptic than the global gamble $f$, because he is always able to eventually turn the initial capital $\martingale{}(\Box)=\alpha$ into a capital that is higher than the uncertain payoff corresponding to $f$, simply by choosing the right gambles from the ones Forecaster is offering.
% at least, if he gambles in a correct way (that is, according to $\martingale{}$).
That $\martingale{}$ should be bounded below, represents the condition that Skeptic can borrow at most a finite amount.

For any $f\in\gambles{}$, the infimum over all the hedging prices $\alpha$ is then what defines the (unconditional) global game-theoretic upper expectation $\upprevvovk{\smash{\overline{\mathrm{Q}}}}(f)$ of $f$.
More generally, the \emph{global game-theoretic upper expectation} of any $f\in\gambles{}$ conditional on any $s\in\situations{}$, is defined as 
% the lower enevelope over all the prices that allow Skeptic to hedge $f$ on the paths that go through the situation $s$:
\begin{multline}\label{Eq: definition real-valued game-theoretic upper expectation}
\upprevvovk{\smash{\overline{\mathrm{Q}}}}(f\vert s)\coloneqq\inf\big\{\martingale{}(s)\colon \martingale{}\in\setofsupmartb{}(\smash{\overline{\mathrm{Q}}}), \\
 (\forall\omega\in\Gamma(s))\liminf \martingale{}(\omega)\geq f(\omega)\big\},
\end{multline}
where $\setofsupmartb{}(\smash{\overline{\mathrm{Q}}})$ denotes the set of all bounded below supermartingales.
The unconditional case corresponds to $s=\Box$; so, $\upprevvovk{\smash{\overline{\mathrm{Q}}}}(f)\coloneqq\upprevvovk{\smash{\overline{\mathrm{Q}}}}(f\vert \Box)$.

As the attentive reader may have noticed, the definition above only applies to global gambles.
So why not to general variables $f\in\extvariables{}$?
The reason is that, on this extended domain, the formula presented above would yield an upper expectation with rather weak continuity properties \cite[section~8]{Tjoens2020FoundationsARXIV}.
A simple solution is to use continuity with respect to so-called upper and lower cuts to extend the domain from $\gambles{}\times\situations{}$ to $\extvariables{}\times\situations{}$.\footnote{This is similar to how  \cite[Chapter~15]{troffaes2014} extends the notion of coherence from gambles to unbounded real-valued variables.}
% not by applying the formula in Equation~\eqref{Eq: definition real-valued game-theoretic upper expectation}, 
To do so, for any $f\in\extvariables{}$ and any $c\in\reals{}$, let $f^{\wedge c}$ be defined by $f^{\wedge c}(x)\coloneqq\min\{f(x),c\}$ for all $x\in\statespace{}$, and let $f^{\vee c}$ be defined analogously, as a pointwise maximum.
Then we henceforth let $\upprevvovk{\smash{\overline{\mathrm{Q}}}}\colon\extvariables{}\times\situations{}\to\extreals{}$ be defined by Equation~\eqref{Eq: definition real-valued game-theoretic upper expectation} on $\gambles{}\times\situations{}$, and furthermore impose, for any $s\in\situations{}$, that 
\begin{enumerate}[leftmargin=*,ref={\upshape{G\arabic*}},label={\upshape{}G\arabic*}.,itemsep=-0pt, series=sepcoherence]
\item \label{global upper cuts} \(\upprevvovk{\smash{\overline{\mathrm{Q}}}}(f\vert s) = \lim_{c\to+\infty} \upprevvovk{\smash{\overline{\mathrm{Q}}}}(f^{\wedge c}\vert s)\) for all \(f\in\bextvariables{}\);
\item \label{global lower cuts} \(\upprevvovk{\smash{\overline{\mathrm{Q}}}}(f\vert s)=\lim_{c\to-\infty} \upprevvovk{\smash{\overline{\mathrm{Q}}}}(f^{\vee c}\vert s)\) for all \(f\in\extvariables{}\).
\end{enumerate}
Properties~\ref{global upper cuts} and~\ref{global lower cuts} together clearly imply that $\upprevvovk{\smash{\overline{\mathrm{Q}}}}$ is uniquely determined by its values on $\gambles{}\times\situations{}$.
Hence, since $\upprevvovk{\smash{\overline{\mathrm{Q}}}}$ on this domain is described by Equation~\eqref{Eq: definition real-valued game-theoretic upper expectation}, it follows that $\upprevvovk{\smash{\overline{\mathrm{Q}}}}$ is uniquely defined on all of $\extvariables{}\times\situations{}$.

This way of extending a global game-theoretic upper expectation is not that common, though.
A technique that is used more often consists in directly applying Equation~\eqref{Eq: definition real-valued game-theoretic upper expectation} to the entire domain $\extvariables{}\times\situations{}$, but with the real-valued supermartingales replaced by extended real-valued ones \cite{Vovk2019finance,TJOENS202130,Tjoens2020FoundationsARXIV}.
This of course first requires an extension $\extlupprev{s}$ of the local models $\lupprev{s}$ to the domain $\smash{\genvariables{}(\statespace{})}$, which can be done in a way similar to what we have done with $\upprevvovk{\smash{\overline{\mathrm{Q}}}}$, by imposing continuity with respect to upper and lower cuts.\footnote{The choice of extending the local models $\lupprev{s}$ in this particular way, by imposing continuity with respect to upper and lower cuts, is motivated in \cite[Sections~2 and~8]{Tjoens2020FoundationsARXIV} and \cite[section~6]{TJOENS202130}, and is, as far as the resulting global game-theoretic upper expectation---with extended real-valued supermartingales---is concerned, completely equivalent with how Shafer and Vovk axiomatise their local models in \cite[Part~II]{Vovk2019finance}.}
An extended real-valued supermartingale $\martingale{}\colon\situations{}\to\extreals{}$ is then characterised by the condition that $\extlupprev{s}(\martingale{}(s \cdot)) \leq \martingale{}(s)$ for all $s\in\situations{}$.
Remarkably enough, the global game-theoretic upper expectation that results from this `extended supermartingale'-approach is identical to the operator $\upprevvovk{\smash{\overline{\mathrm{Q}}}}$ we have defined above, using Properties~\ref{global upper cuts} and~\ref{global lower cuts}; see for example the end of \cite[Section~8]{Tjoens2020FoundationsARXIV}.
We favor our approach, though, because the use of extended real-valued supermartingales undermines what we think is a key strength of the game-theoretic approach: 
that supermartingales---and hence the resulting game-theoretic upper expectations---can be given a clear behavioural meaning in terms of betting.

\subsection{Axiomatic Global Upper Expectations}
\label{Sect: Axiomatic Global Upper Expectations}

Instead of relying on measure-theoretic or game-theoretic principles, one can also simply adopt an abstract global model $\upprev{}$ that is completely characterised by a number of axioms.
In particular, starting from any given upper expectation tree $\overline{\mathrm{Q}}$, we suggest to impose the following list of axioms:
\begin{enumerate}[leftmargin=*,ref={\upshape P\arabic*},label={\upshape P\arabic*}.,itemsep=1pt,series=Properties]
\item \label{P compatibility}
$\upprev (f(X_{n+1}) \vert \sit{}) = \lupprev{\sit{}}(f)$ for all $f \in \gengambles{}(\statespace{})$ and \\ all $\sit{} \in \situations{}$.
% For any $f \in \gengambles{}(\statespace{})$ and any $\sit{} \in \situations{}$,\vspace{-2pt}
% \begin{equation*}
% 		\upprev (f(X_{n+1}) \vert \sit{}) = \lupprev{\sit{}}(f).\vspace{-2pt}
% \end{equation*}
\item \label{P conditional} 
% \begin{equation*}
		$\upprev (f \vert s) = \upprev (f \, \indica{s} \vert s)$ for all $f \in \fingambles{}$ and all $s \in \situations{}$.
% \end{equation*}
\item \label{P iterated} 
% For all $f \in \fingambles{}$ and all $k \in \natz$, 
% \vspace{-2pt}
% \begin{align*}
$\upprev(f \vert X_{1:k}) \leq \upprev(\upprev(f \vert X_{1:k+1}) \vert X_{1:k})$ for all $f \in \fingambles{}$ and \\all $k \in \natz$.
% \vspace{-2pt}
% \end{align*}
\item \label{P monotonicity} 
$f \leq g \Rightarrow \upprev(f \vert s) \leq \upprev(g \vert s)$ for all $f,g \in \extvariables{}$ and all $s\in\situations{}$.
% \vspace{-2pt}
% \begin{align*}
% f \leq g \Rightarrow \upprev(f \vert s) \leq \upprev(g \vert s).\vspace{-2pt}
% \end{align*}
\item \label{P continuity}
For any sequence \((f_n)_{n \in \nats}\) of finitary gambles that is uniformly bounded below and any \(s \in \situations{}\):
\begin{equation*}
\lim_{n \to +\infty} f_n = f \Rightarrow 
\limsup_{n \to +\infty} \upprev{}(f_n \vert s) \geq \upprev{}(f \vert s).\vspace*{-4pt}
\end{equation*}
\end{enumerate}
\noindent
Here, as well as further on, we call a sequence $(f_n)_{n \in \nats{}}$ of variables uniformly bounded below if there is some $B\in\reals{}$ such that $f_n(\omega) \geq B$ for all $n \in \nats{}$ and $\omega \in \samplespace{}$.
Furthermore, the limit $\lim_{n \to +\infty} f_n$, as well as all others in this paper, are intended to be taken pointwise.

Axioms~\ref{P compatibility}--\ref{P continuity} are put forward here because, as we argue in \cite[Section~4]{TJOENS202130}, they can be motivated on the basis of two different interpretations for a global upper expectation; a direct behavioural interpretation in terms of minimum selling prices, or a probabilistic interpretation in terms of sets of linear expectations (or probability measures).
Basically, we find Axioms~\ref{P compatibility}--\ref{P monotonicity} straightforward and believe them to be almost unquestionable, regardless of the adopted interpretation.
Axiom~\ref{P continuity}, which imposes a form of continuity, is perhaps more disputable.
Nonetheless, compared to other well-known continuity properties, such as dominated convergence or monotone convergence, Property~\ref{P continuity} is rather weak because it only applies to sequences of finitary gambles.
Note that, in general, finitary gambles play a central role in our axiomatisation; with the exception of monotonicity (Axiom~\ref{P monotonicity}), all our axioms exclusively apply to finitary gambles (and their limits).
We find this important because, as explained in \cite[Section~4]{TJOENS202130}, they are the only global variables that we feel can be given a direct operational meaning, and hence, the only global variables for which axioms can be motivated directly.
More general global variables in $\extvariables{}$, on the other hand, that depend on an infinite number of state values, or are unbounded or even infinite-valued, should be regarded as abstract idealisations.

Of course, even if we agree upon Axioms~\ref{P compatibility}--\ref{P continuity}, it does not necessarily provide us with a global upper expectation because there may be multiple---or, worse, no---global upper expectations satisfying these axioms.
The following result shows that there is at least one model that satisfies \ref{P compatibility}--\ref{P continuity}, and that among all the ones that satisfy them, there is a unique most conservative---that is, largest---one.
We denote this model by $\axupprev{\smash{\overline{\mathrm{Q}}}}$.

% Hence, we are still left with the question of which of these upper expectations to choose---if there only exists one.
% To that end, we prefer to take the largest one.
% The reason that we do so is because, for both of the interpretations described above, a larger upper expectation corresponds to a more conservative---less informative---uncertainty model; it corresponds to a higher selling price or to a larger set of possible expectations or probabilities.
% Hence, if we choose to adopt Axioms~\ref{P compatibility}--\ref{P continuity}, but not necessarily want to add any more `information' into our global model, the best thing to do is to pick the largest upper expectation $\axupprev{\smash{\overline{\mathrm{Q}}}}$ among all the global upper expectations $\setofupprev{}(\smash{\overline{\mathrm{Q}}})$ that satisfy \ref{P compatibility}--\ref{P continuity} for a given upper expectation tree $\smash{\overline{\mathrm{Q}}}$.
% The following theorem guarantees that there always is such a (unique) largest one.
\begin{theorem}[{\cite[Theorem~6]{TJOENS202130}}]\label{theorem: Vovk is the largest}
% The set \/ \(\setofupprev{}_{1-5}(\tree{})\) is non-empty. 
For any upper expectation tree $\smash{\overline{\mathrm{Q}}}$, there is a unique most conservative global upper expectation \(\axupprev{\smash{\overline{\mathrm{Q}}}}\) that satisfies \ref{P compatibility}--\ref{P continuity}.
% the set \/ \(\setofupprev{}(\smash{\overline{\mathrm{Q}}})\) is non-empty and it contains a unique most conservative upper expectation \(\axupprev{\smash{\overline{\mathrm{Q}}}}\).
\end{theorem}

\section{An Equality for Monotone Limits of Finitary Gambles}\label{Sect: An Equality for Monotone Limits of Finitary Gambles}

Having introduced all three global upper expectations, we can finally turn to the central problem of this paper: 
how are these upper expectations related to each other?
More specifically, we ask ourselves the following.
If the parameters of a stochastic process are equivalent---that is, if the trees $\probtree{}$ and $\smash{\overline{\mathrm{Q}}}$ agree---are the global models $\mupprev{\probtree{}}$, $\upprevvovk{\smash{\overline{\mathrm{Q}}}}$ and $\axupprev{\smash{\overline{\mathrm{Q}}}}$ then equal?
In a recent paper \cite{TJOENS202130}, we have shown that the answer is affirmative for the latter two models.
\begin{theorem}[{\cite[Theorem~6]{TJOENS202130}}]\label{theorem: vovk and axiomatic are equal}
The global upper expectations\/ $\axupprev{\smash{\overline{\mathrm{Q}}}}$ and\/ $\upprevvovk{\smash{\overline{\mathrm{Q}}}}$ are equal.
\end{theorem}

So it only remains to study the relationship between the measure-theoretic upper expectation $\mupprev{\probtree{}}$ and the common upper expectation $\upprev{}_{\smash{\overline{\mathrm{Q}}}}\coloneqq\axupprev{\smash{\overline{\mathrm{Q}}}}=\upprevvovk{\smash{\overline{\mathrm{Q}}}}$.
To do so, we will build on two earlier results, gathered from that same paper \cite{TJOENS202130};
the first one \cite[Theorem~14]{TJOENS202130} says that $\mupprev{\probtree{}}$ coincides with $\upprev{}_{\smash{\overline{\mathrm{Q}}}}$ if the tree $\probtree$ is a precise probability tree $p$ (and $\smash{\overline{\mathrm{Q}}}$ is the agreeing (upper) expectation tree); 
the second one \cite[Proposition~21]{TJOENS202130} says that they are also equal for general imprecise probability trees $\probtree{}$, provided that we limit ourselves to finitary gambles.
Our main results extend this equality for general imprecise probability trees in two ways: 
to variables that are monotone limits of finitary gambles and to bounded below $\mathscr{F}$-measurable variables.
In the current section, we work towards establishing the first extension.
Our approach is straightforward;
we will prove that $\mupprev{\probtree{}}$ and $\upprev{}_{\smash{\overline{\mathrm{Q}}}}$ are both continuous with respect to monotone sequences of finitary gambles.
Since they coincide on finitary gambles, this directly implies the desired equality.

We start by showing that $\mupprev{\probtree{}}$ and $\upprev{}_{\smash{\overline{\mathrm{Q}}}}$ are both continuous with respect to non-decreasing sequences in $\bextvariables{}$---and hence definitely in $\fingambles{}$.
\begin{proposition}\label{prop: upward continuity measure-theoretic}
For any\/ $\probtree{}$ and\/ $\smash{\overline{\mathrm{Q}}}$, any $s\in\situations{}$ and any non-decreasing sequence $(f_n)_{n\in\nats{}}$ in $\bextvariables{}$, we have that \vspace*{-2pt}
\begin{equation*}
\lim_{n\to+\infty}\mupprev{\probtree{}}(f_n \vert s) = \mupprev{\probtree{}}(f \vert s) \text{, with } f = \sup_{n\in\nats{}} f_n = \lim_{n\to+\infty} f_n,\vspace*{-2pt}
\end{equation*}
and similarly for\/ $\upprev{}_{\smash{\overline{\mathrm{Q}}}}$. 
\end{proposition}
\begin{proof}
That the statement holds for $\upprev{}_{\smash{\overline{\mathrm{Q}}}}$ follows immediately from \cite[Theorem~9(i)]{TJOENS202130}.
To prove the statement for $\mupprev{\probtree{}}$, recall \cite[Theorem~14]{TJOENS202130}, which says that, for any precise probability tree $p$ and the agreeing (upper) expectation tree $\smash{\overline{\mathrm{Q}}_p}$, we have that $\smash{\upprev[p](g\vert s) = \upprev[\smash{\overline{\mathrm{Q}}_p}](g\vert s)}$ for all $g\in\extvariables{}$.
Then, since $\smash{\upprev[\smash{\overline{\mathrm{Q}}_p}]}$ is continuous with respect to non-decreasing sequences in $\bextvariables{}$ \cite[Theorem 9(i)]{TJOENS202130}, we have, for any precise probability tree $p$, that 
% \begin{equation*}
$\lim_{n\to+\infty}\upprev[p](f_n \vert s) = \upprev[p](f \vert s)$.
% \end{equation*}
Hence, it follows that 
\begin{equation*}
\sup_{p\sim\probtree{}}\lim_{n\to+\infty}\upprev[p](f_n \vert s) = \sup_{p\sim\probtree{}} \upprev[p](f \vert s) = \mupprev{\probtree{}}(f \vert s).
\end{equation*}
On the other hand, we also have that
\begin{align*}
\sup_{p\sim\probtree{}}\lim_{n\to+\infty}\upprev[p](f_n \vert s) 
&\leq \lim_{n\to+\infty}\sup_{p\sim\probtree{}}\upprev[p](f_n \vert s) \\
&= \lim_{n\to+\infty}\mupprev{\probtree{}}(f_n \vert s)
,
\end{align*}
where the two last limits exist because $(f_n)_{n\in\nats}$ is non-decreasing and $\mupprev{p}$---and therefore also $\mupprev{\probtree{}}$---is monotone\iftoggle{arxiv}{; see e.g. Lemma~\ref{lemma: global properties} in Appendix~\ref{Sect: Appendix B}}{; see e.g. \cite[Lemma~18]{TJoens2021Equivalence}}.
So we obtain that $\smash{\mupprev{\probtree{}}(f \vert s) \leq \lim_{n\to+\infty}\mupprev{\probtree{}}(f_n \vert s)}$. 
The converse inequality follows from the fact that $f_n \leq f$ for all $n\in\nats{}$ and the monotonicity of $\mupprev{\probtree{}}$.
\end{proof}

% Note that any imprecise probability tree $\probtree{}$ can always equivalently be represented by a set of precise probability trees, but not the other way around.
% For example, an imprecise Markov chain under complete (or repetition) independence can be represented by a set of precise probability trees---which are themselves Markov (and time-homogeneous)---but not by a single imprecise probability tree.

\noindent
Next, we prove that $\mupprev{\probtree{}}$ is also continuous with respect to non-increasing sequences in $\fingambles{}$---that $\upprev{}_{\smash{\overline{\mathrm{Q}}}}$ satisfies this type of continuity was already established in \cite[Theorem~9(ii)]{TJOENS202130}.
The proof is less straightforward, though, and first requires us to establish the following two topological lemmas concerning probability trees. 
We will say that a sequence $(p_i)_{i\in\nats}$ of precise probability trees converges if there is some limit tree $p$ such that, for each $s\in\situations{}$, the mass functions $(p_i(\cdot\vert s))_{i\in\nats{}}$ converge (pointwise) to the mass function $p(\cdot \vert s)$.

\begin{lemma}\label{lemma: convergent subsequence of probability trees}
Consider any imprecise probability tree $\probtree{}$.
Then any sequence $(p_i)_{i\in\nats{}}$ of precise probability trees that are compatible with $\probtree{}$ has a converging subsequence whose limit is compatible with $\probtree{}$.
\end{lemma}

\begin{lemma}\label{lemma: convergence of probability measures implies convergence on n-measurables}
Consider any sequence $(p_i)_{i\in\nats}$ of precise probability trees that converges to some limit tree $p_{}$.
Then, for any $g\in\fingambles{}$ and any $s\in\situations{}$,
% \begin{equation*}
$\lim_{i\to+\infty}\mprev{p_i}(g \vert s)=\mprev{p_{}}(g \vert s)$.
% \end{equation*}
\end{lemma}

Combined, the two lemmas above suffice to prove the continuity of $\mupprev{\probtree{}}$ with respect to non-increasing sequences in $\fingambles{}$.

\begin{proposition}\label{Prop: decreasing continuity measure-theoretic}
For any\/ $\probtree{}$ and\/ $\smash{\overline{\mathrm{Q}}}$, any $s\in\situations{}$ and any non-increasing sequence $(f_n)_{n\in\nats{}}$ of finitary gambles, we have that \vspace*{-2pt}
\begin{equation*}
\lim_{n\to+\infty}\mupprev{\probtree{}}(f_n \vert s) = \mupprev{\probtree{}}(f \vert s) \text{, with } f =\inf_{n\in\nats{}} f_n = \lim_{n\to+\infty} f_n,\vspace*{-2pt}
\end{equation*}
and similarly for\/ $\upprev{}_{\smash{\overline{\mathrm{Q}}}}$.
\end{proposition}
\begin{proof}
The statement for $\upprev{}_{\smash{\overline{\mathrm{Q}}}}$ follows from \cite[Theorem~9(ii)]{TJOENS202130}.
To prove the statement for $\mupprev{\probtree{}}$ first note that, since $(f_n)_{n\in\nats{}}$ is non-increasing and all $f_n$ are gambles, the variable $f$ is bounded above.
Moreover, $f$ is $\mathscr{F}$-measurable because it is a pointwise limit of finitary---and therefore certainly $\mathscr{F}$-measurable---gambles \cite[Theorem II.4.2]{shiryaev2016probabilityPartI}.
Taking both facts into account, we deduce that, for any $p\sim\probtree{}$, the expectation $\mprev{p}(f\vert s)$ exists and hence, because $\mupprev{p}$ is an extension of $\mprev{p}$ (see Section~\ref{Sect: Measure-Theoretic Global Upper Expectations}), that $\mupprev{p}(f\vert s) = \mprev{p}(f\vert s)$.
Since this obviously also holds for each $f_n$---because they are finitary and bounded---the desired statement follows if we manage to show that\vspace*{-2pt}
\begin{equation*}
\lim_{n\to+\infty} \sup \Bigl\{\mprev{p}(f_n \vert s) \colon p\sim\probtree \Bigr\} = \sup \Bigl\{\mprev{p}( f \vert s) \colon p\sim\probtree \Bigr\}.\vspace*{-2pt}
\end{equation*}
The `$\geq$'-inequality  follows immediately from the fact that $f_m \geq \inf_{n\in\nats{}}f_n = f$ for all $m\in\nats{}$ and the monotonicity of $\mprev{p}$. 
It remains to prove the converse inequality.

Fix any $\epsilon>0$ and let $(p_i)_{i\in\nats{}}$ be a sequence of precise probability trees such that $p_i \sim\probtree$ and\vspace*{-2pt}
\begin{equation*}
\mprev{p_i}(f_i \vert s) + \epsilon \geq \sup \bigl\{\mprev{p}( f_i \vert s) \colon p\sim\probtree \bigr\} \text{ for all } i\in\nats{}.\vspace*{-2pt}
\end{equation*}
Note that this is indeed possible because, for all $i\in\nats{}$, $\sup \bigl\{\mprev{p}( f_i \vert s) \colon p\sim\probtree \bigr\}\leq \sup f_i$ and, since $f_i$ is a gamble, $\sup f_i \in\reals{}$.
Then Lemma~\ref{lemma: convergent subsequence of probability trees} guarantees that $(p_{i})_{i\in\nats{}}$ has a convergent subsequence $(p_{i(k)})_{k\in\nats{}}$ whose limit $p^\ast$ is compatible with $\probtree{}$.
% Let $p^\ast\sim\probtree{}$ be the limit of this convergent subsequence.
Since $\mprev{p^\ast}$ satisfies continuity with respect to non-increasing sequences \cite[Property~M9]{TJOENS202130} (the required conditions are obviously satisfied because $f_n$ is finitary and $f_n \leq f_1 \leq \sup f_1 \in\reals{}$ for all $n\in\nats{}$), there is, for any real $a>\mprev{p^\ast}(f \vert s)$, some $n^\ast\in\nats$ such that
% \begin{equation}\label{Eq: Prop: decreasing continuity measure-theoretic 1}
$a \geq \mprev{p^\ast}(f_{n^\ast} \vert s)$.
% \end{equation}
% Fix any $n\geq n^\ast$.
Furthermore, since $f_{n^\ast}$ is finitary, and since $\mprev{p^\ast}(f_{n^\ast} \vert s)\in\reals{}$ because $f_{n^\ast}$ is a gamble, Lemma~\ref{lemma: convergence of probability measures implies convergence on n-measurables} implies that there is some $k^\ast\in\nats$ such that 
% \begin{equation*}
$\smash{\mprev{p_{i(k)}}(f_{n^\ast} \vert s)-\epsilon\leq \mprev{p^\ast}(f_{n^\ast} \vert s)} \text{ for all } k\geq k^\ast$.
% \end{equation*}
We therefore get that
\begin{equation}\label{Eq: Prop: decreasing continuity measure-theoretic 2}
a \geq \mprev{p_{i(k)}}(f_{n^\ast} \vert s) - \epsilon \text{ for all } k\geq k^\ast.
\end{equation}
Now consider any $k\geq k^\ast$ such that $i(k)\geq n^\ast$, which is possible because $(i(k))_{k\in\nats}$ is increasing.
Then, since $(f_n)_{n\in\nats}$ is non-increasing, and $\mprev{p_{i(k)}}$ is monotone, Equation~\eqref{Eq: Prop: decreasing continuity measure-theoretic 2} implies that 
% \begin{equation}
$a \geq \mprev{p_{i(k)}}(f_{i(k)} \vert s) - \epsilon$.
% \end{equation}
Since the tree $p_{i(k)}$ was chosen in such a way that 
% \begin{equation*}
$\smash{\mprev{p_{i(k)}}(f_{i(k)} \vert s) +\epsilon \geq \sup \bigl\{\mprev{p}( f_{i(k)} \vert s) \colon p\sim\probtree \bigr\}}$, this implies that
% \end{equation*}
% Hence, we have that 
% \begin{equation}
$a \geq \sup \bigl\{\mprev{p}( f_{i(k)} \vert s) \colon p\sim\probtree \bigr\} - 2\epsilon$.
% \end{equation}
Because this holds for any $k\geq k^\ast$ such that $i(k)\geq n$, we find that
\begin{align*}
a 
&\geq \lim_{k\to+\infty} \sup \bigl\{\mprev{p}( f_{i(k)} \vert s) \colon p\sim\probtree \bigr\} - 2\epsilon \\
&= \lim_{n\to+\infty} \sup \bigl\{\mprev{p}( f_{n} \vert s) \colon p\sim\probtree \bigr\} - 2\epsilon,
\end{align*}
where the equality follows from the fact that $(f_n)_{n\in\nats{}}$ is non-increasing and the monotonicity of $\mprev{p}$.
Since this holds for any real $a>\mprev{p^\ast}(f \vert s)$, it follows that
$\smash{\mprev{p^\ast}(f \vert s)
\geq \lim_{n\to+\infty} \sup \bigl\{\mprev{p}( f_{n} \vert s) \colon p\sim\probtree \bigr\}}-2\epsilon$.
Finally, it suffices to recall that $p^\ast \sim\probtree{}$, to see that
\begin{equation*}
\sup \Bigl\{\mprev{p}( f \vert s) \colon p\sim\probtree \Bigr\} 
\geq \lim_{n\to+\infty} \sup \bigl\{\mprev{p}( f_{n} \vert s) \colon p\sim\probtree \bigr\} -2\epsilon,
\end{equation*} 
which, since $\epsilon>0$ was arbitrary, concludes the proof.
\end{proof}

It now remains to combine the two types of continuity with the fact that $\mupprev{\probtree{}}$ and $\upprev{}_{\smash{\overline{\mathrm{Q}}}}$ coincide on $\fingambles{}\times\situations{}$ \cite[Proposition~21]{TJOENS202130} to arrive at our first main result.
% observe that, by \cite[Theorem~9]{TJOENS202130}, $\upprev{}_{\smash{\overline{\mathrm{Q}}}}$ satisfies the same continuity properties as the ones that we have just established for $\mupprev{\probtree{}}$ with Proposition~\ref{prop: upward continuity measure-theoretic} and~\ref{Prop: decreasing continuity measure-theoretic},
% combine this with the fact that these two models coincide on finitary gambles \cite[Proposition~21]{TJOENS202130}, to arrive at our first major result.

\begin{theorem}\label{theorem: equivalence for non-decreasing limits of n-measurables}
Consider any $\probtree{}$ and $\smash{\overline{\mathrm{Q}}}$ that agree, any $s\in\situations{}$ and any $f\in\extvariables{}$ that is the pointwise limit of a non-decreasing or non-increasing sequence of finitary gambles.
Then we have that\/ 
% \begin{equation*}
$\mupprev{\probtree{}}(f\vert s) = \upprev{}_{\smash{\overline{\mathrm{Q}}}}(f\vert s)$.
% \text{ and } \ \mlowprev{\probtree{}}(f\vert s) = \lowprev{}_{\,\smash{\overline{\mathrm{Q}}}}(f\vert s).
% \end{equation*}
\end{theorem}
% A formal proof of it can be found in Appendix~\ref{Sect: Appendix A}.

\section{An Equality for $\mathscr{F}$-Measurable Variables}\label{Sect: equivalence for measurable variables} 
\noindent
In order to prove our second main result---that $\mupprev{\probtree{}}$ coincides with $\upprev{}_{\smash{\overline{\mathrm{Q}}}}$ on bounded below $\mathscr{F}$-measurable variables---we require the notions of upper and lower semicontinuity.

Let $\samplespace{}$ be endowed with the topology generated by the cylinder events $\{\Gamma(s)\colon s\in\situations{}\}$.
As we show in \iftoggle{arxiv}{Appendix~\ref{Sect: Appendix B}}{\cite[Appendix A.2]{TJoens2021Equivalence}}, this topology is metrisable and compact, and coincides with the product topology on $\samplespace{}=\statespace{}^\nats{}$.
% According to Lemma~\ref{lemma: metric topology is product topology} in Appendix~\ref{Sect: Appendix B}, 
% Then, as we explain in Appendix~\ref{Sect: Appendix B}, the topological space $\samplespace{}$ is metrisable and compact, and the cylinder events $\{\Gamma(s)\colon s\in\situations{}\}$ form a subbase.
For any topological space $\posspace{}$---and hence for $\samplespace{}$ in particular---a function $f\colon\posspace{}\to\extreals{}$ is called \emph{upper semicontinuous (u.s.c.)} if $\{y\in\posspace{}\colon f(y) < a\}$ is an open subset of $\posspace{}$ for each $a\in\reals{}$; see \cite[Section~11.C and~23.F]{kechris1995set_theory} or \cite[Section~3.7.K]{willard2004general}.
A function $f\colon\posspace{}\to\extreals{}$ is called \emph{lower semicontinuous (l.s.c.)} if $-f$ is u.s.c. and it is called \emph{continuous} if it is both u.s.c. and l.s.c.
% A fundamental property of semicontinuous functions is that, if they are defined on a metrisable space, they can always be written as the pointwise limit of a monotone sequence of continuous real-valued functions; see e.g. \cite[Theorem~23.19]{kechris1995set_theory}.
% For our case where $\posspace{}=\samplespace{}$, however, this result can further be strengthened by replacing continuous real-valued functions with $n$-measurable variables (or functions), the latter clearly being continuous.
In general, semicontinuous functions can always be written as pointwise limits of monotone sequences of continuous real-valued functions (see e.g. \cite[Theorem~23.19]{kechris1995set_theory}).
In our case, though, where $\mathscr{Y}=\samplespace{}$, a stronger property holds.

\begin{lemma}\label{lemma: semicontinuous functions are monotone limits of n-measurables}
Any $f\in\extvariables$ is u.s.c. (l.s.c.) if and only if it is the pointwise limit of a non-increasing (resp. non-decreasing) sequence $(f_n)_{n\in\nats{}}$ of extended real variables, each of which is finitary and bounded below (resp. bounded above).
Moreover, $f$ is both u.s.c. (l.s.c.) and bounded above (resp. bounded below) if and only if it is the pointwise limit of a non-increasing (resp. non-decreasing) sequence $(f_n)_{n\in\nats{}}$ of finitary gambles.
\end{lemma}

Lemma~\ref{lemma: semicontinuous functions are monotone limits of n-measurables} leads us to two important intermediate results, the first of which being that $\mupprev{\probtree{}}$ and $\upprev{}_{\smash{\overline{\mathrm{Q}}}}$ coincide on the domain of all u.s.c. variables that are bounded above and all l.s.c. variables that are bounded below.
The result can simply be seen as a restatement of Theorem~\ref{theorem: equivalence for non-decreasing limits of n-measurables}\iftoggle{arxiv}{ and is therefore stated without proof.}{.}

\begin{corollary}\label{corollary: equivalence for upper semicontinuous functions}
For any\/ $\probtree{}$ and\/ $\smash{\overline{\mathrm{Q}}}$ that agree, any $s\in\situations{}$ and any variable $f\in\extvariables{}$ that is u.s.c. and bounded above, or l.s.c. and bounded below, we have that\/ 
% \begin{equation*}
$\mupprev{\probtree{}}(f\vert s) = \upprev{}_{\smash{\overline{\mathrm{Q}}}}(f\vert s)$.
% \end{equation*}
\end{corollary}
% \begin{proof}
% Consider any $s\in\situations{}$ and any u.s.c. variable $f\in\extvariables{}$ that is bounded above.
% According to Lemma~\ref{lemma: semicontinuous functions are monotone limits of n-measurables}, $f$ is the pointwise limit of a non-increasing sequence $(f_n)_{n\in\nats{}}$ of finitary gambles.
% % If we define $(g_n)_{n\in\natz{}}$ as the sequence obtained by bounding $(f_n)_{n\in\natz{}}$ from above by $\sup f \in\reals{}$---i.e. $g_n(\omega)\coloneqq \min\{f_n(\omega), \sup f\}$ for all $n\in\natz{}$ and all $\omega\in\samplespace{}$---then $(g_n)_{n\in\natz{}}$ is clearly again a non-increasing sequence of $n$-measurable variables that converges pointwise to $f$, but it is moreover a sequence of gambles.
% Then, subsequently applying Proposition~\ref{Prop: decreasing continuity measure-theoretic}, \cite[Proposition~21]{TJOENS202130} and Proposition~\ref{Prop: decreasing continuity measure-theoretic}, yields\vspace*{-2pt}
% \begin{align*}
% \mupprev{\probtree{}}(f \vert s)
% = \lim_{n\to+\infty} \mupprev{\probtree{}}(f_n \vert s)
% &= \lim_{n\to+\infty} \upprev{}_{\smash{\overline{\mathrm{Q}}}}(f_n \vert s) 
% = \upprev{}_{\smash{\overline{\mathrm{Q}}}}(f \vert s),\vspace*{-4pt}
% \end{align*}
% therefore concluding the proof.
% \end{proof}

On the other hand, Lemma~\ref{lemma: semicontinuous functions are monotone limits of n-measurables} also implies that continuity with respect to non-increasing sequences of (bounded above) u.s.c. variables is actually not stronger than continuity with respect to non-increasing sequences of finitary gambles; see \iftoggle{arxiv}{Lemma~\ref{lemma: upper semicont continuity} in Appendix~\ref{Sect: Appendix B}}{\cite[Lemma~17]{TJoens2021Equivalence}}.
Since both $\mupprev{\probtree{}}$ and $\upprev{}_{\smash{\overline{\mathrm{Q}}}}$ satisfy the latter type of continuity, we immediately obtain the following result.
% since $\mupprev{\probtree{}}$ and $\upprev{}_{\smash{\overline{\mathrm{Q}}}}$ both satisfy continuity with respect to non-increasing sequences of finitary gambles, both of them also satisfy continuity with respect to non-increasing sequences of upper semicontinuous variables that are bounded above.
\begin{proposition}\label{prop: upper semicont continuity}
Consider any $\probtree{}$ and $\smash{\overline{\mathrm{Q}}}$, any $s\in\situations{}$ and any non-increasing sequence $(f_n)_{n\in\nats{}}$ of u.s.c. variables that are bounded above.
Then we have that 
% \begin{equation*}
$\lim_{n\to+\infty}\mupprev{\probtree{}}(f_n \vert s) = \mupprev{\probtree{}}(f \vert s)$ for $f=\lim_{n\to+\infty} f_n$, and similarly for $\upprev{}_{\smash{\overline{\mathrm{Q}}}}$.
 % \lim_{n\to+\infty}\upprev{}_{\smash{\overline{\mathrm{Q}}}}(f_n \vert s) = \upprev{}_{\smash{\overline{\mathrm{Q}}}}(f \vert s).
% \end{equation*}
% $\lim_{n\to+\infty}\mupprev{\probtree{}}(f_n \vert s) = \mupprev{\probtree{}}(f \vert s)$ and $\lim_{n\to+\infty}\upprev{}_{\smash{\overline{\mathrm{Q}}}}(f_n \vert s) = \upprev{}_{\smash{\overline{\mathrm{Q}}}}(f \vert s)$ for $f=\lim_{n\to+\infty}f_n$.
% \text{ and } \\
% \lim_{n\to+\infty}\upprev{}_{\smash{\overline{\mathrm{Q}}}}(f_n \vert s) &= \upprev{}_{\smash{\overline{\mathrm{Q}}}}(f \vert s)
% \end{align*}
\end{proposition}
\begin{proof}
This follows from \iftoggle{arxiv}{Lemma~\ref{lemma: upper semicont continuity} in Appendix~\ref{Sect: Appendix B}}{\cite[Lemma~17]{TJoens2021Equivalence}}, Proposition~\ref{Prop: decreasing continuity measure-theoretic} and the fact that $\mupprev{\probtree{}}$ and $\upprev{}_{\smash{\overline{\mathrm{Q}}}}$ are clearly both monotonous (see \iftoggle{arxiv}{Lemma~\ref{lemma: global properties} in Appendix~\ref{Sect: Appendix B}}{\cite[Lemma~18]{TJoens2021Equivalence}}).
\end{proof}
Note that, conversely, $\mupprev{\probtree{}}$ and $\upprev{}_{\smash{\overline{\mathrm{Q}}}}$ are also continuous with respect to non-decreasing sequences of l.s.c. variables that are bounded below, simply because, due to Proposition~\ref{prop: upward continuity measure-theoretic}, they satisfy continuity with respect to any non-decreasing (bounded below) sequence.

As a final step towards establishing our desired result, we will use what is called Choquet's capacitability theorem.
% to establish an equality between $\mupprev{\probtree{}}$ and $\upprev{}_{\smash{\overline{\mathrm{Q}}}}$ on the domain of all variables that are $\mathscr{F}$-measurable and bounded below.
This theorem can be found in many different textbooks, but we will make use of the specific version of \citet{Dellacherie1972_ensembles_analytiques}.
We do this because Dellacherie's notion of a capacity can directly be applied to an extended real-valued functional---such as $\mupprev{\probtree{}}$ and $\upprev{}_{\smash{\overline{\mathrm{Q}}}}$---whereas most other sources restrict capacities to take the form of set-functions.
Let us start by introducing some key concepts and terminology regarding capacitability and analytic functions.

Let $\nnegextvariables$ be the set of all variables taking values in $\nnegextreals{}$ and $\nnegvariables$ the set of all (possibly unbounded) variables taking values in $\nnegreals{}$.
A functional $\mathrm{F}\colon\nnegextvariables\to\nnegextreals{}$ is called a \emph{$\samplespace{}$-capacity} if it satisfies the following three properties \cite[Section II.1.1]{Dellacherie1972_ensembles_analytiques}:\vspace*{-1pt}
\begin{enumerate}[leftmargin=*,ref={\upshape{CA\arabic*}},label={\upshape{}CA\arabic*}.,itemsep=-0pt, series=capacity]

\item\label{capacity: monotonicity}
$f\leq g \Rightarrow \mathrm{F}(f)\leq\mathrm{F}(g)$ for all $f,g\in\nnegextvariables$;

\item\label{capacity: upward continuity}
$\lim_{n\to+\infty} \mathrm{F}(f_n)=\mathrm{F}\left( \lim_{n\to+\infty} f_n \right)$ for any non-decreasing sequence $(f_n)_{n\in\nats}$ in \/ $\nnegextvariables$;

\item\label{capacity: downward continuity}
$\lim_{n\to+\infty} \mathrm{F}(f_n)=\mathrm{F}\left( \lim_{n\to+\infty} f_n \right)$ for any non-increasing sequence $(f_n)_{n\in\nats}$ of u.s.c. variables in\/~$\nnegvariables$.
\vspace*{-1pt}
\end{enumerate}
Recall from the beginning of this section that $\samplespace{}$ is compact and metrisable, which is in line with Dellacherie's assumption about the set `$\mathrm{E}$' in \cite[Section II.1.1]{Dellacherie1972_ensembles_analytiques}; see \cite[Introduction, Paragraph~2]{Dellacherie1972_ensembles_analytiques}.
Furthermore, observe that \ref{capacity: downward continuity} only applies to sequences in $\nnegvariables$ instead of sequences in $\nnegextvariables{}$; this too corresponds to the definition given in \cite[Section II.1.1]{Dellacherie1972_ensembles_analytiques} because Dellacherie always considers u.s.c. functions to be real-valued \cite[Introduction, Paragraph~2]{Dellacherie1972_ensembles_analytiques}.
In fact, one could restate \ref{capacity: downward continuity} so as to only apply to sequences that are uniformly bounded above; this follows immediately from the non-increasing character and the following lemma.\vspace*{-0pt}
\begin{lemma}\label{lemma: real upper semicont functions are bounded above}
Any u.s.c. variable $f\in\nnegvariables{}$ is bounded above.\vspace*{-0pt}
\end{lemma}

For any $\samplespace{}$-capacity $\mathrm{F}$, we say that a variable $f\in\nnegextvariables{}$ is \emph{$\mathrm{F}$-capacitable} if 
% $\mathrm{F}(f)=\sup\mathrm{F}(g)$ with $g\in\nnegvariables{}$ being u.s.c. and $f\geq g$.
\vspace*{-1pt}
\begin{equation}\label{Eq: capacitability}
\mathrm{F}(f)=\sup\bigl\{\mathrm{F}(g)\colon g\in\nnegvariables{}\text{, $g$ is u.s.c.} \text{ and } f\geq g \bigr\}.\vspace*{-1pt}
\end{equation}
A variable $f\in\nnegextvariables{}$ is called \emph{universally capacitable} if it is $\mathrm{F}$-capacitable for all $\samplespace{}$-capacities $\mathrm{F}$.
Now, Choquet's capacitability theorem \cite[Theorem II.2.5]{Dellacherie1972_ensembles_analytiques} states that any \emph{analytic} variable is universally capacitable.
The definition of an analytic variable can be found in \cite{Dellacherie1972_ensembles_analytiques,kechris1995set_theory}; 
we do not explicitly give it here, because it is a rather abstract concept that, in practice, can often be replaced by the simpler and better-known notion of a Borel-measurable variable.
Indeed, according to \cite[Section~I.2.6]{Dellacherie1972_ensembles_analytiques}, each Borel-measurable variable in $\nnegextvariables{}$ is analytic.
Moreover, by \iftoggle{arxiv}{Corollary~\ref{Corollary: Borel sigma algebra} in Appendix~\ref{Sect: Appendix B}}{\cite[Corollary~16]{TJoens2021Equivalence}}, the Borel $\sigma$-algebra on $\samplespace{}$ coincides with the $\sigma$-algebra $\mathscr{F}$ generated by all cylinder events, so 
% note that the Borel $\sigma$-algebra on $\samplespace{}$---the $\sigma$-algebra generated by all open sets in $\samplespace{}$---coincides with the $\sigma$-algebra $\mathscr{F}$ generated by all cylinder events (this follows for instance from Lemma~\ref{lemma: metric topology is product topology} in Appendix~\ref{Sect: Appendix B}, which says that open sets in $\samplespace{}$ are countable unions of cylinder events and, conversely, that cylinder events are open sets).
the notions of Borel-measurability and $\mathscr{F}$-measurability are equivalent. 
Combined with \cite[Theorem II.2.5]{Dellacherie1972_ensembles_analytiques}, this allows us to state the following weaker version of Choquet's capacitability theorem:
 
\begin{theorem}[Choquet's capacitability light]\label{Theorem: Choquet}
Any $\mathscr{F}$-measurable variable $f\in\nnegextvariables{}$ is universally capacitable.
\end{theorem}

As an almost immediate consequence of Proposition~\ref{prop: upward continuity measure-theoretic}, Proposition~\ref{Prop: decreasing continuity measure-theoretic} and Lemma~\ref{lemma: real upper semicont functions are bounded above}, it can be shown that, for any $s\in\situations{}$, the restrictions of both $\mupprev{\probtree{}}(\cdot\vert s)$ and $\upprev{}_{\smash{\overline{\mathrm{Q}}}}(\cdot\vert s)$ to $\nnegextvariables{}$ are $\samplespace{}$-capacities\iftoggle{arxiv}{; see Appendix~\ref{Sect: Appendix B}}{ \cite[Appendix A.2]{TJoens2021Equivalence}}.
Therefore, and because these upper expectations coincide on the u.s.c. variables in $\nnegvariables{}$---due to Corollary~\ref{corollary: equivalence for upper semicontinuous functions} and Lemma~\ref{lemma: real upper semicont functions are bounded above} above---the desired equality for $\mathscr{F}$-measurable variables in $\nnegextvariables{}$ follows from Equation~\eqref{Eq: capacitability} and Theorem~\ref{Theorem: Choquet}.
We can moreover replace $\nnegextvariables{}$ by $\bextvariables{}$, simply because $\mupprev{\probtree{}}$ and $\upprev{}_{\smash{\overline{\mathrm{Q}}}}$ are linear with respect to adding constants \iftoggle{arxiv}{(see Lemma~\ref{lemma: global properties} in Appendix~\ref{Sect: Appendix B})}{\cite[Lemma~18]{TJoens2021Equivalence}}.
This leads to our second main result.

\begin{theorem}\label{theorem: equivalence for measurable functions}
For any\/ $\probtree{}$ and\/ $\smash{\overline{\mathrm{Q}}}$ that agree, any $s\in\situations{}$ and any $\mathscr{F}$-measurable variable $f\in\bextvariables{}$ that is bounded below, we have that\/ 
% \begin{equation*}
$\mupprev{\probtree{}}(f\vert s) = \upprev{}_{\smash{\overline{\mathrm{Q}}}}(f\vert s)$.
% \end{equation*}
% If $f$ is moreover bounded above, and therefore a gamble, we also have that\, $\mlowprev{\probtree{}}(f\vert s) = \lowprev{}_{\,\smash{\overline{\mathrm{Q}}}}(f\vert s)$.
\end{theorem}

% Similarly, an \emph{SV-imprecise probability tree} \/ $\exttree \colon s\in\situations{}\mapsto \extlupprev{s}$ maps situations $s\in\situations{}$ to SV-upper expectations $\extlupprev{s}\colon\domain{}\to\extreals{}$ with $\bgenvariables{}(\posspace{})\subseteq\domain{}\subseteq\genvariables{}(\posspace{})$.
% Finally, a \emph{precise probability tree} $\precprobtree\colon s\in\situations{}\mapsto\lprob{s}$ is a function that maps each situation $s\in\situations{}$ to a probability mass function $\lprob{s}$ on $\statespace{}$.
% It follows from the considerations in the previous section that each imprecise probability tree $\tree{}$ can equivalently be represented by a set of precise probability trees

\section{Relation with Shafer and Vovk's Work}\label{Sect: Relation with Shafer and Vovk}

Before we conclude this paper, it seems appropriate to say a few words about how our work here compares to that of Shafer and Vovk.
As readers that are familiar with their work may have noticed, the idea to use Choquet's capacitability theorem to extend the domain of the equality to $\mathscr{F}$-measurable (or analytic) variables already appears in \cite[Chapter~9]{Vovk2019finance}.
Another part that strongly builds on ideas from \cite[Chapter~9]{Vovk2019finance} is the proof of Proposition~\ref{Prop: decreasing continuity measure-theoretic};
some key steps there were inspired by the proof of \cite[Lemma~9.10]{Vovk2019finance}.
 % some of the essential ideas behind our results---such as the reasoning in the proof of Proposition~\ref{Prop: decreasing continuity measure-theoretic} and the use of Choquet's capacitability theorem---already appear in \cite[Section~9]{Vovk2019finance}, which, in general, served as an important inspiration for our work. 
So it is fair to say that \cite[Chapter~9]{Vovk2019finance} served as an important inspiration for our work.
In fact, to the untrained eye, it might perhaps even seem as if our results do not differ much from those in \cite[Chapter~9]{Vovk2019finance}; but take a closer look.

First of all---and most importantly---the setting in which we define game-theoretic upper expectations differs considerably from theirs.
More specifically, they consider supermartingales under the prequential principle, which says that Forecaster's moves---the specification of the local models $\lupprev{s}$ (or $\extlupprev{s}$)---are not necessarily known beforehand for each situation $s\in\situations{}$, but instead are allowed to also depend on previous moves by Skeptic; see \cite[Theorem~7.5]{Vovk2019finance} for more details.
While this assumption allows them to remain more general---though, in many practical cases, it does not make much of a difference---the benefit that we gain from dropping it is remarkable;
it allows us to replace \cite[Lemma~9.10]{Vovk2019finance} and \cite[Theorem~9.7]{Vovk2019finance}, which require strong topological conditions on the parametrisation of the local models, with respectively Theorem~\ref{theorem: equivalence for non-decreasing limits of n-measurables} and Theorem~\ref{theorem: equivalence for measurable functions}, which are similar, but do not need any topological conditions at~all.

A second notable difference is that our results involve a larger domain;
Theorem~\ref{theorem: equivalence for non-decreasing limits of n-measurables}, or equivalently, Corollary~\ref{corollary: equivalence for upper semicontinuous functions}, applies to both u.s.c. variables that are bounded above and l.s.c. variables that are bounded below, whereas \cite[Lemma~9.10]{Vovk2019finance} only applies to bounded u.s.c. variables; 
Theorem~\ref{theorem: equivalence for measurable functions} applies to bounded below ($\mathscr{F}$-measurable) variables, whereas \cite[Theorem~9.7]{Vovk2019finance} only applies to bounded (analytic) variables.\footnote{
Recall that we could just as well have stated Theorem~\ref{theorem: equivalence for measurable functions} for analytic variables instead of $\mathscr{F}$-measurable variables.}
Our results also allow conditioning on situations;
% and we have included (conditional) lower expectations as well; 
theirs only apply to unconditional upper expectations.
The fact that this extension in domain is relevant in practice becomes clear when we also take a look at lower expectations.
Indeed, in (more) practical situations, we are usually not only interested in the upper expectation of a variable, but also, and simultaneously, in its lower expectation \cite{8627473,8535240}.
Our results can be easily extended to this two-sided setting, by combining the conjugacy relation between global upper and lower expectations with our two main results.
 % an equality that simultaneously applies to both global upper and lower expectations.

\begin{corollary}\label{corollary: equivalence}
Consider any\/ $\probtree{}$ and\/ $\smash{\overline{\mathrm{Q}}}$ that agree, any $s\in\situations{}$ and any $f\in\extvariables{}$ that is \emph{(a)} the pointwise limit of a monotone sequence of finitary gambles or \emph{(b)} an $\mathscr{F}$-measurable gamble.
Then we have that 
\begin{equation*}
\mupprev{\probtree{}}(f\vert s) = \upprev{}_{\smash{\overline{\mathrm{Q}}}}(f\vert s) \text{ and } \, \mlowprev{\probtree{}}(f\vert s) = \lowprev{}_{\,\smash{\overline{\mathrm{Q}}}}(f\vert s).
\end{equation*}
% If $f$ is moreover bounded above, and therefore a gamble, we also have that\, $\mlowprev{\probtree{}}(f\vert s) = \lowprev{}_{\,\smash{\overline{\mathrm{Q}}}}(f\vert s)$.
\end{corollary}

\noindent
Note that many practically relevant inferences---e.g. hitting times \cite{8627473}---fall under category (a) but not under category (b), simply because they are not bounded. 
Yet, it is exactly this class of variables that is missing in Shafer and Vovk's main result \cite[Theorem~9.7]{Vovk2019finance}.

Finally, recall that our results relate $\mupprev{\probtree{}}$ to $\upprev{}_{\smash{\overline{\mathrm{Q}}}}$, where the latter represents, apart from the game-theoretic upper expectation $\upprevvovk{\smash{\overline{\mathrm{Q}}}}$, also the axiomatic upper expectation $\axupprev{\smash{\overline{\mathrm{Q}}}}$. 
Shafer and Vovk, on the other hand, only relate $\mupprev{\probtree{}}$ to the game-theoretic upper expectation $\upprevvovk{\smash{\overline{\mathrm{Q}}}}$.

% \textcolor{blue}{Also say something about the fact that they only prove that $\mupprev{\probtree{}}$ and $\upprev{}_{\smash{\overline{\mathrm{Q}}}}$ are capacities on $[0,C]$ and that they capacitiability theorem that needs stronger conditions? And that they do not give definition of Suslin. don't say anything about n-measurables.}

\section{Conclusion}\label{Sect: Conclusion}

Our main results, Theorem~\ref{theorem: equivalence for non-decreasing limits of n-measurables} and Theorem~\ref{theorem: equivalence for measurable functions}, show that measure-theoretic, game-theoretic and axiomatic upper expectations are equal on a large domain of variables; it contains all variables that are the limit of a monotone sequence of finitary gambles, and all variables that are bounded below and $\mathscr{F}$-measurable.
It remains to be seen whether we can extend this equivalence even further, to all variables; 
so far, we have yet to find a counterexample showing that this is not possible.

We would also like to investigate the relation between our models and the Daniell-Stone type of (global) upper expectations described in \cite{Denk2018}.
Comparing Theorem~\ref{theorem: equivalence for measurable functions} and \cite[Theorem~3.10]{Denk2018}, and taking into account their use of Choquet's capacitability theorem, it seems that a close connection must exist, at least for bounded measurable variables. A more thorough study is required though before we can make accurate statements.

% , where even weaker types of relations were used to establish valuable theoretical insights, as well as several more practically oriented results.

\acks{
The research of Natan T'Joens was supported and funded by the Special Research Fund (BOF) of Ghent University (reference number: 356).
The research of Jasper De Bock was partially funded by project number 3GO28919 of the FWO (Research Foundation - Flanders). We thank the reviewers for their thorough reading of our manuscript.
}

\bibliographystyle{plain}
\bibliography{ISIPTA2021}

\iftoggle{arxiv}{
\appendix
\section{Proofs}
\subsection{Proofs for the Results in Section~\ref{Sect: An Equality for Monotone Limits of Finitary Gambles}}\label{Sect: Appendix A}

Let $\mathbb{P}_{\statespace{}}$ be the set of all probability mass functions on $\statespace{}$ and let $d$ be the total variation distance \cite[Section~7.1]{DeBock:2015ck} defined, for any two mass functions $\pi_1,\pi_2\in\mathbb{P}_{\statespace{}}$, by
\begin{equation}\label{Eq: def distance between mass functions}
d(\pi_1,\pi_2) \coloneqq \max_{A\subseteq\statespace{}} \vert \pi_1(A)-\pi_2(A) \vert 
=  \frac{1}{2} \sum_{x\in\statespace{}} \vert \pi_1(x)-\pi_2(x) \vert,
\end{equation}
where we allowed ourselves a slight abuse of notation by writing $\pi_i(A)$ to mean $\sum_{x\in A} \pi_i(x)$ for $i\in\{1,2\}$.
Let $\mathbb{P}_\statespace{}$ be endowed with the topology induced by $d$, which is equivalent---see \cite[Appendix A]{DeBock:2015ck}---to the topology of pointwise convergence that we have implicitly adopted in the main text.
So $\mathbb{P}_{\statespace{}}$ is metrisable and, by \cite[Section~7]{DeBock:2015ck}, compact.
Also, note that any precise probability tree $p\colon s\in\situations{}\mapsto p(\cdot\vert s)\in\mathbb{P}_{\statespace{}}$ can be regarded as an element of the product space $\bigtimes_{s\in\situations{}} \mathbb{P}_{\statespace{}}$, and any imprecise probability tree $\probtree{}$ can be seen as a subset of $\bigtimes_{s\in\situations{}} \mathbb{P}_{\statespace{}}$.
% $\prod_{s\in\situations{}} \setofprob[s]\subseteq\bigtimes_{s\in\situations{}} \mathbb{P}_{\statespace{}}$.
Saying that a precise probability tree $p$ is compatible with an imprecise probability tree $\probtree{}$ is then the same as saying that $p\in\probtree{}$.
We will moreover endow the space $\bigtimes_{s\in\situations{}} \mathbb{P}_{\statespace{}}$ with the product topology.
It is clear that a sequence of precise probability trees $(p_i)_{i\in\nats{}}$ then converges if, for each situation $s\in\situations{}$, the mass functions $(p_i(\cdot\vert s))_{i\in\nats{}}$ converge pointwise, which is in accordance with our assumptions in the main text. 

\begin{proofof}[Lemma~\ref{lemma: convergent subsequence of probability trees}]
For any $s\in\situations{}$, since the credal set $\probtree_{s}$ is a closed and convex subset of the compact space $\mathbb{P}_{\statespace{}}$, it follows that $\probtree_{s}$ is compact.
Therefore, by Tychonoff's theorem \cite[Theorem 17.8]{willard2004general}, the tree $\probtree{}$ is compact too (as a subset of $\bigtimes_{s\in\situations{}} \mathbb{P}_{\statespace{}}$).
 % as a subset of $\bigtimes_{s\in\situations{}} \mathbb{P}_{\statespace{}}$.
Moreover, note that, due to \cite[Theorem 22.3]{willard2004general} and the metrizability of $\setofprob{}_{\statespace{}}$ (and the fact that $\situations{}$ is countable), the space $\bigtimes_{s\in\situations{}} \mathbb{P}_{\statespace{}}$ is also metrisable.
So, by \cite[17G.3.]{willard2004general}, the compactness of $\probtree{}$ implies its sequential compactness.
Hence, by definition, each sequence in $\probtree{}$ has a convergent subsequence whose limit belongs to $\probtree{}$. 
\end{proofof}

\begin{proofof}[Lemma~\ref{lemma: convergence of probability measures implies convergence on n-measurables}]
First of all, observe that, for all $i\in\nats{}$, the expectations $\mprev{p_{}}(g \vert s)$ and $\mprev{p_i}(g \vert s)$ are indeed well-defined---the corresponding Lebesgue integrals exist---because $g$ is bounded and finitary (and therefore certainly $\mathscr{F}$-measurable).
In fact, because we can write $g = \sum_{z_{1:\ell}\in\statespace{}^\ell} g(z_{1:\ell}) \indica{z_{1:\ell}}$ for some $\ell\in\nats$, these expectations simply reduce---by definition of the Lebesgue integral; see \cite[Section 2.6.1]{shiryaev2016probabilityPartI}---to the finite weighted sums 
\begin{align}\label{Eq: lemma: convergence of probability measures implies convergence on n-measurables 1}
\vspace*{-2pt}
\mprev{p_{}}(g \vert s) &= \sum_{z_{1:\ell}\in\statespace{}^\ell} g(z_{1:\ell}) \prob_{p_{}}(z_{1:\ell} \vert s) \vspace*{-6pt}
\end{align}
and\vspace*{-4pt}
\begin{align}\label{Eq: lemma: convergence of probability measures implies convergence on n-measurables 2}
\mprev{p_i}(g \vert s) &= \sum_{z_{1:\ell}\in\statespace{}^\ell} g(z_{1:\ell}) \prob_{p_i}(z_{1:\ell} \vert s), \vspace*{-2pt}
\end{align}
where $\prob_{p_{}}(\cdot\vert s)$ and $\prob_{p_i}(\cdot\vert s)$ are defined according to Equation~\eqref{Eq: precise probability on algebra}.
Let $x_{1:k}\in\situations{}$ be such that $s=x_{1:k}$ and fix any $z_{1:\ell}\in\statespace{}^\ell$.
% , then the probability $\prob_{p_i}(z_{1:\ell} \vert s)$ for all $i\in\natz{}$ is equal to
% \begin{align*}
% \prob_{p_i}(z_{1:\ell}\vert x_{1:m}) =
% \begin{aligned}
% \begin{cases}
% \prod_{\,\ell=m}^{n-1} p_i( x_{\ell+1} \vert x_{1:\ell}) 
% &\text{ if } m < n \text{ and } x_{1:m} = x_{1:m} \\
%  1   &\text{ if } m \geq n \text{ and } z_{1:\ell} = z_{1:\ell} \\
%  0 &\text{ otherwise, }
% \end{cases}
% \end{aligned}
% \end{align*}
% and analogously for the probability $\prob_{p}(z_{1:\ell} \vert s)$.
We will now show that $\prob_{p_i}(z_{1:\ell}\vert s)$ converges to $\prob_{p_{}}(z_{1:\ell} \vert s)$ as a function of $i\in\nats$.
% The desired convergence will then follows from Equation~\eqref{Eq: lemma: convergence of probability measures implies convergence on n-measurables 1} together with, on the one hand, the fact that the coefficients $g(z_{1:\ell})$ for all $z_{1:\ell}\in\statespace{}^n$ are real---because $g$ is a gamble---and, on the other hand, the finiteness of $\statespace$ 

If $k \geq \ell$ and $z_{1:\ell}=x_{1:\ell}$, then $\prob_{p_i}(z_{1:\ell}\vert s)=1$ for all $i\in\nats{}$ and also $\prob_{p_{}}(z_{1:\ell}\vert s)=1$, so $\prob_{p_i}(z_{1:\ell}\vert s)$ surely converges to $\prob_{p_{}}(z_{1:\ell} \vert s)$.
Similar observations lead us to conclude that this is also true for the cases where, either, $k \geq \ell$ and $z_{1:\ell}\not= x_{1:\ell}$, or, $k < \ell$ and $z_{1:k} \not= x_{1:k}$.
So it remains to check whether it is true for the case where $k < \ell$ and $z_{1:k} = x_{1:k}$.
In that case, $\prob_{p_i}(z_{1:\ell}\vert x_{1:k}) = \prod_{\,n=k}^{\ell-1} p_i( z_{n+1} \vert z_{1:n})$ converges to $\prob_{p_{}}(z_{1:\ell}\vert x_{1:k}) = \prod_{\,n=k}^{\ell-1} p\,( z_{n+1} \vert z_{1:n})$ if, for all $n \in\{k,\cdots,\ell-1\}$, $p_i( z_{n+1} \vert z_{1:n})$ converges to $p\,(z_{n+1} \vert z_{1:n})$.
The latter is implied by the convergence of $p_i$ to $p_{}$.
Indeed, since $\bigtimes_{s\in\situations{}}\mathbb{P}_{\statespace{}}$ is equipped with the product topology, the convergence of $p_i$ to $p_{}$ implies that, for any $n \in\{k,\cdots,\ell-1\}$, the mass function $p_i(\cdot\,\vert z_{1:n})$ converges to $p\,(\cdot\,\vert z_{1:n})$.
Since the set $\mathbb{P}_{\statespace{}}$ on its turn is equipped with the topology of pointwise convergence, this implies that
% induced by the metric $d$, so $d\big(p_i(\cdot\,\vert z_{1:n}),p\,(\cdot\,\vert z_{1:n})\big)$ converges to zero and therefore, by Equation~\eqref{Eq: def distance between mass functions}, 
$p_i(z_{n+1} \vert z_{1:n})$ converges to $p\,(z_{n+1}\vert z_{1:n})$.

Now, to conclude the proof, note that the sums in Equations~\eqref{Eq: lemma: convergence of probability measures implies convergence on n-measurables 1} and~\eqref{Eq: lemma: convergence of probability measures implies convergence on n-measurables 2} are over a finite set $\statespace{}^\ell$---because $\statespace{}$ is finite---and the coefficients $g(z_{1:\ell})$ are real because $g$ is a gamble.
Since we have just shown that, for any $z_{1:\ell}\in\statespace{}^\ell$, the probability $\prob_{p_i}(z_{1:\ell}\vert s)$ converges to $\prob_{p_{}}(z_{1:\ell} \vert s)$, it is therefore clear that the expectation $\mprev{p_i}(g \vert s)$ converges to $\mprev{p_{}}(g \vert s)$.
\end{proofof}

\begin{proofof}[Theorem~\ref{theorem: equivalence for non-decreasing limits of n-measurables}]
Suppose that $f$ is the pointwise limit of a non-decreasing sequence of finitary gambles.
Then we have that $\mupprev{\probtree{}}(f\vert s) = \upprev{}_{\smash{\overline{\mathrm{Q}}}}(f\vert s)$ because, on the one hand, $\mupprev{\probtree{}}$ coincides with $\upprev{}_{\smash{\overline{\mathrm{Q}}}}$ for all finitary gambles \cite[Proposition~21]{TJOENS202130}, and on the other hand, due to Proposition~\ref{prop: upward continuity measure-theoretic}, both $\mupprev{\probtree{}}$ and $\upprev{}_{\smash{\overline{\mathrm{Q}}}}$ are continuous with respect to non-decreasing sequences of gambles.
Suppose now that $f$ is the pointwise limit of a non-increasing sequence $(f_n)_{n\in\nats}$ of finitary gambles.
Then similarly, the desired equality follows from \cite[Proposition~21]{TJOENS202130} and Proposition~\ref{Prop: decreasing continuity measure-theoretic}.
\end{proofof}

\subsection{Proofs for the Results in Section~\ref{Sect: equivalence for measurable variables}}\label{Sect: Appendix B}

Consider the distance function $\delta$ on $\samplespace{}$ defined by
\begin{equation}\label{Eq: distance between paths}
\delta(\omega,\omega')\coloneqq 2^{-n} \text{ with } n\coloneqq\inf\,\{k\in\nats{}\colon \omega_{k}\not=\omega'_k\}, 
\end{equation}
for all $\omega,\omega'\in\samplespace{}$.
Then it can easily be checked that $\delta$ is a metric on $\samplespace{}$.
Furthermore, as is shown by the lemma below, the topology on $\samplespace{}$ corresponding to this metric $\delta$ is the same as the topology that we have adopted throughout the main text---that is, the topology generated by the cylinder events $\{\Gamma(s)\colon s\in\situations{}\}$.
This confirms our claim that $\samplespace{}$ is metrisable.
Moreover, the lemma below also shows that this metric topology coincides with the product topology and therefore, by Tychonoff's theorem \cite[Theorem 17.8]{willard2004general} and the finiteness of $\statespace{}$ (and therefore the compactness of $\statespace{}$), that $\samplespace{}$ is compact.
% product topology that we have adopted in the main text; indeed, as is shown by the lemma below, both of these topologies are generated by the cylinder events $\{\Gamma(s)\colon s\in\situations{}\}$.
% Since we have endowed $\samplespace{}$ with the product topology, it follows that $\samplespace{}$ is metrisable.
% Furthermore, $\samplespace{}$ is also compact due to Tychonoff's theorem \cite[Theorem 17.8]{willard2004general} and the fact that $\statespace{}$ is finite (and therefore compact).  

\begin{lemma}\label{lemma: metric topology is product topology}
The set\/ $\{\Gamma(s)\colon s\in\situations{}\}$ of all cylinder events is a subbase for the metric topology on $\samplespace{}$ corresponding to~$\delta$.
The same holds for the product topology on $\samplespace{}$, and hence, the metric topology and product topology coincide.
Moreover, a set in this topology is open if and only if it is a countable union of cylinder events.
\end{lemma}
\begin{proof}
% [Lemma~\ref{lemma: metric topology is product topology}]
Recall that the set of all open $\epsilon$-disks form a subbase for the metric topology; see e.g. \cite[Example 3.2(a)]{willard2004general}.
Consider any such open $\epsilon$-disk; that is, for any $\epsilon>0$ and any $\omega\in\samplespace{}$, consider the set $\{\omega'\in\samplespace{}\colon \delta(\omega,\omega') < \epsilon\}$.
If $\epsilon>1$, let $\ell\coloneqq 0$; otherwise, let $\ell\in\natz{}$ be the unique natural number such that $2^{-\ell-1} < \epsilon \leq 2^{-\ell}$.
Then, for all $\omega'\in\Gamma(\omega^\ell)$, since $\inf\{k\in\nats{}\colon \omega'_k\not=\omega_k\} \geq \ell+1$, we have by Equation~\eqref{Eq: distance between paths} that $\delta(\omega,\omega') \leq 2^{-\ell-1}<\epsilon$.
On the other hand, for any $\omega'\not\in\Gamma(\omega^\ell)$, we infer in a similar way that $\delta(\omega,\omega') \geq 2^{-\ell}\geq\epsilon$.
Hence, both facts taken together, we obtain that $\Gamma(\omega^\ell)=\{\omega'\in\samplespace{}\colon \delta(\omega,\omega') < \epsilon\}$ is the open $\epsilon$-disk around $\omega$.
Conversely, one can see that any cylinder event $\Gamma(x_{1:\ell})$ with $x_{1:\ell}\in\situations{}$, is an open $\epsilon$-disk around any $\omega\in\Gamma(x_{1:\ell})$ if $\epsilon>0$ is such that $2^{-\ell-1} < \epsilon \leq 2^{-\ell}$.
As a consequence, the family of open $\epsilon$-disks in $\samplespace{}$ is the same as the set $\{\Gamma(s)\colon s\in\situations{}\}$ of all cylinder events and therefore, since the former is a subbase of the metric topology, the set $\{\Gamma(s)\colon s\in\situations{}\}$ is a subbase of the metric topology.
This establishes the first statement.

Let us show that the same holds for the product topology on $\samplespace{}=\statespace{}^\nats{}$.
Since $\statespace{}$ has the discrete topology, the sets $U_{n,y}\coloneqq\{\omega\in\samplespace{}\colon \omega_n = y\}$ with $n\in\nats{}$ and $y\in\statespace{}$ form a subbase of this topology \cite[Definition 8.3]{willard2004general}.
Clearly, any such set $U_{n,y}$ is the union of the cylinder events $\Gamma(x_{1:n-1} y)$ with $x_{1:n-1}\in\statespace{}^{n-1}$, so the topology generated by the cylinder events $\{\Gamma(s)\colon s\in\situations{}\}$ is finer than (includes) the product topology.
On the other hand, any cylinder event $\Gamma(x_{1:n})$ with $x_{1:n}\in\situations{}$ is the finite intersection of the sets $U_{i,x_i}$ with $i\in\{1,\cdots,n\}$, so we also have that the product topology is finer than the one generated by $\{\Gamma(s)\colon s\in\situations{}\}$.
All together, we conclude that the topology generated by the cylinder events $\{\Gamma(s)\colon s\in\situations{}\}$ coincides with the product topology---and hence $\{\Gamma(s)\colon s\in\situations{}\}$ is a subbase---which establishes the second statement.

It remains to prove the last statement, which says that a set in this common topology is open if and only if it is a countable union of cylinder events.
In other words, we have to prove that $\tau \coloneqq \{\cup_{i\in\nats{}} \Gamma(s_i)\colon (\forall i\in\nats)\, s_i\in\situations{}\}$ is the topology generated by the subbase $\{\Gamma(s)\colon s\in\situations{}\}$.
That $\tau$ is closed under arbitrary unions follows from the fact that the set $\situations{}$ of all situations is countable.
Indeed, any union of elements of $\tau$ is a union of cylinder events, and since $\situations{}$---and therefore also $\{\Gamma(s)\colon s\in\situations{}\}$---is countable, this union can always be written as a countable union, therefore implying that it is an element of $\tau$.
Now, consider any finite intersection $\cap_{j\in\{1,\cdots,n\}} \cup_{i\in\nats{}} \Gamma(s_{i,j})$ of elements of $\tau$ and let us check that this too is an element of $\tau$.
Using distributivity, the finite intersection $\cap_{j\in\{1,\cdots,n\}} \cup_{i\in\nats{}} \Gamma(s_{i,j})$ can be rewritten as a countable union of finite intersections of cylinder events $\Gamma(s_{i,j})$.
So we can conclude that this countable union is an element of $\tau$ if we manage to show that any finite intersection of cylinder events is itself a cylinder event.
In order to do so, consider the intersection of any two cylinder events $\Gamma(x_{1:n})$ and $\Gamma(y_{1:m})$ with $x_{1:n}\in\situations{}$ and $y_{1:m}\in\situations{}$.
Note that this intersection is non-empty if and only if, either, $n\leq m$ and $x_{1:n}=y_{1:n}$, or, if $n>m$ and $x_{1:m}=y_{1:m}$.
In the first case, we have that $\Gamma(x_{1:n})\cap\Gamma(y_{1:m})=\Gamma(y_{1:m})$ and, in the second case, we have that $\Gamma(x_{1:n})\cap\Gamma(y_{1:m})=\Gamma(x_{1:n})$.
Hence, the intersection of any two cylinder events is itself a cylinder event and therefore, any finite intersection of cylinder events is also a cylinder event.
By our previous considerations, this implies that $\tau$ is indeed closed under finite intersections.
Together with the fact that $\tau$ is closed under arbitrary unions---and trivially includes $\samplespace{}$ and the empty set $\emptyset$---we may conclude that $\tau$ is a topology on $\samplespace{}$.
Since $\{\Gamma(s)\colon s\in\situations{}\}$ is clearly a subbase of this topology $\tau$, this finalises the proof.
\end{proof}
The last statement in the lemma above immediately implies the following corollary, in which the Borel sets are the open sets with respect to the common topology from Lemma~\ref{lemma: metric topology is product topology}.

\begin{corollary}\label{Corollary: Borel sigma algebra}
The Borel $\sigma$-algebra on $\samplespace{}$ coincides with the $\sigma$-algebra $\mathscr{F}$ generated by the cylinder events.
\end{corollary}
\begin{proof}
By Lemma~\ref{lemma: metric topology is product topology}, any open set in $\samplespace{}$ is the countable union of cylinder events.
As a result, all open sets are included in the $\sigma$-algebra $\mathscr{F}$ and therefore, $\mathscr{F}$ includes the Borel $\sigma$-algebra.
On the other hand, it is clear that $\mathscr{F}$ is not larger than the Borel $\sigma$-algebra because each cylinder event is itself open (because it is a---trivial---union of cylinder events).
\end{proof}

\begin{proofof}[Lemma~\ref{lemma: semicontinuous functions are monotone limits of n-measurables}]
Since $-f\in\extvariables{}$ is l.s.c. if and only if $f$ is u.s.c., it clearly suffices to prove the statement for u.s.c. variables.
We start by proving the two direct implications.
Let $f\in\extvariables$ be u.s.c. and let $(f_n)_{n\in\nats{}}$ be defined by
\begin{equation*}
f_n(\omega) \coloneqq \sup(\{-n\}\cup\{f(\omega') \colon \omega'\in\Gamma(\omega^n)\}),
\end{equation*}
for all $\omega\in\samplespace{}$ and all $n\in\nats{}$.
Then $(f_n)_{n\in\nats{}}$ is clearly a non-increasing sequence of variables that are finitary and bounded below (since $f_n \geq -n$).
If $f$ is bounded above, then each $f_n$ is clearly also bounded above, so in that case $(f_n)_{n\in\nats{}}$ is a sequence of gambles.  
So it only remains to show that $\lim_{n\to+\infty}f_n(\omega) = f(\omega)$ for any $\omega\in\samplespace{}$. 
That $\lim_{n\to+\infty}f_n(\omega) \geq f(\omega)$ holds, follows from the fact that, due to the definition of the variables $f_n$, $f_n(\omega)\geq f(\omega)$ for all $n\in\nats{}$.
Hence, if $f(\omega)=+\infty$, we automatically have that $\lim_{n\to+\infty}f_n(\omega) = f(\omega)$, so we may assume that $f(\omega)<+\infty$.
Fix any real $a>f(\omega)$.
Since $f$ is u.s.c., the set $\{\omega'\in\samplespace{}\colon f(\omega') < a\}$ is an open neighboorhood of $\omega$.
According to Lemma~\ref{lemma: metric topology is product topology}, any open set in $\samplespace{}$ is a countable union of cylinder events.
Since $\omega$ belongs to $\{\omega'\in\samplespace{}\colon f(\omega') < a\}$, one of these cylinder events contains $\omega$.
This implies that there is some $n\in\nats{}$ such that $f(\omega') < a$ for all $\omega'\in \Gamma(\omega^n)$.
Then, for any $k\geq n$, since $\Gamma(\omega^k)\subseteq\Gamma(\omega^n)$, we obviously also have that $f(\omega') < a$ for all $\omega'\in\Gamma(\omega^k)$.
Hence, $f_k(\omega) \leq a$ for all $k \geq \max\{\abs{a},n\}$, which implies that $\lim_{k\to+\infty}f_k(\omega) \leq a$.
This holds for any real $a>f(\omega)$, so we obtain that $\lim_{k\to+\infty}f_k(\omega) \leq f(\omega)$ as desired.
% Now consider the case where $f(\omega)=-\infty$ (the case where $f(\omega)=+\infty$ is impossible because $f$ is bounded above).
% Since $f$ is upper semicontinuous, the set $\{\omega'\in\samplespace{}\colon f(\omega') < a\}$ for any $a\in\reals{}$ is an open neighboorhood of $\omega$.
% Similarly as before, this implies that, for any $a\in\reals{}$, there is some $n\in\nats{}$ such that $f(\omega') < a$ for all $\omega'\in\Gamma(\omega^k)$ and all $k\geq n$.
% This in turn implies that $f_k(\omega) \leq a$ for all $k \geq \max\{\abs{a},n\}$, and therefore that $\lim_{k\to+\infty}f_k(\omega) \leq a$.
% Since this holds for any $a\in\reals{}$, we conclude that $\lim_{k\to+\infty}f_k(\omega)=-\infty= f(\omega)$.

To prove the two converse implications, consider any $f\in\extvariables$ that is the pointwise limit of a non-increasing sequence $(f_n)_{n\in\nats{}}$ of finitary bounded below variables.
We show that, for any $a\in\reals{}$, the set $A\coloneqq\{\omega\in\samplespace{}\colon f(\omega) < a\}$ is open, and therefore that $f$ is a u.s.c. variable.
It is then clear that $f$ is moreover bounded above if $(f_n)_{n\in\nats{}}$ is a sequence of gambles, because in that case $f \leq f_1 \leq \sup f_1 \in\reals{}$. 
So fix any $a\in\reals{}$ and note that the sequence $(A_n)_{n\in\nats{}}$ of events defined by $A_n\coloneqq\{\omega\in\samplespace{}\colon f_n(\omega) < a\}$ for all $n\in\nats{}$, is non-decreasing and converges to $A$ because $(f_n)_{n\in\nats{}}$ converges non-increasingly to $f$.
So we have that $A = \cup_{n\in\nats{}} A_n$.
Moreover, for any $n\in\nats{}$, because $f_n$ is finitary, there is a $k\in\nats$ such that $f_n$ only depends on the first $k$ states, and so the set $A_n$ is a finite union of cylinder events of the form $\Gamma(x_{1:k})$ with $x_{1:k}\in\statespace{}^k$.
So, by Lemma~\ref{lemma: metric topology is product topology}, each set $A_n$ is open.
Since any union of open sets is open again, we obtain that $A = \cup_{n\in\nats{}} A_n$ is open, therefore concluding the proof.
\end{proofof}

\begin{lemma}\label{lemma: upper semicont continuity}
Any operator\/ $\mathrm{F}\colon\extvariables{}\to\extreals{}$ that is monotone and that is continuous with respect to non-increasing (or non-decreasing) sequences of finitary gambles, is also continuous with respect to non-increasing (resp. non-decreasing) sequences $(f_n)_{n\in\nats{}}$ of u.s.c. (resp. l.s.c.) variables that are bounded above (resp. bounded below); i.e. 
\begin{equation*}
\lim_{n\to+\infty}\mathrm{F}(f_n) = \mathrm{F}(f) \text{, with } f = \inf_{n\in\nats{}} f_n = \lim_{n\to+\infty} f_n.
\end{equation*}
\end{lemma}
\begin{proof}
Consider any non-increasing sequence $(f_n)_{n\in\nats{}}$ of u.s.c. variables that are bounded above.
Then it follows from Lemma~\ref{lemma: semicontinuous functions are monotone limits of n-measurables} that, for all $n\in\nats{}$, there is a non-increasing sequence $(g_{n,m})_{m\in\nats}$ of finitary gambles such that $\lim_{m\to+\infty} g_{n,m} = f_n$.
Now let $(h_m)_{m\in\nats{}}$ be the sequence of variables defined by
\begin{equation*}
h_m(\omega)\coloneqq\min\{g_{n,m}(\omega) \colon 0\leq n \leq m\} \text{ for all } \omega\in\samplespace{}.
\end{equation*} 
Because each $(g_{n,m})_{m\in\nats}$ is non-increasing, $(h_m)_{m\in\nats{}}$ is also non-increasing.
The variables $h_m$ for all $m\in\nats{}$ are clearly bounded---and hence, they are gambles---and they are also finitary because, on the one hand, $g_{n,m}$ is finitary for all $n\in\nats{}$, and on the other hand, the minimum over a finite number of finitary variables is trivially also finitary.
So $(h_m)_{m\in\nats{}}$ is a non-increasing sequence of finitary gambles.
Furthermore, note that $h_m \geq f$ because $g_{n,m} \geq f_n \geq f$ for all $n,m\in\nats{}$, and therefore $\lim_{m\to+\infty} h_m \geq f$.
To see that $\lim_{m\to+\infty} h_m \leq f$, fix any $\omega\in\samplespace{}$ and any $a\in\reals{}$ such that $a>f(\omega)$.
Since $\lim_{n\to+\infty} f_n = f$, there is some $n'\in\nats{}$ such that $a>f_{n'}(\omega)$ and since also $\lim_{m\to+\infty} g_{n',m} = f_{n'}$, there is some $m'\geq n'$ such that $a>g_{n',m'}(\omega)$.
Then certainly $a > h_{m'}(\omega)$, and since $(h_m)_{m\in\nats{}}$ is non-increasing, we have that $a > \lim_{m\to+\infty}h_m(\omega)$.
This holds for any $a\in\reals{}$ such that $a>f(\omega)$, so we have that $\lim_{m\to+\infty} h_m(\omega) \leq f(\omega)$, which in turn implies that $\lim_{m\to+\infty} h_m \leq f$ because $\omega\in\samplespace{}$ was chosen arbitrarily.
So we have that $\lim_{m\to+\infty} h_m = \inf_{m\in\nats{}} h_m = f$.
Now, recalling that $(h_m)_{m\in\nats{}}$ is moreover a non-increasing sequence of finitary gambles, it follows from the assumptions about $\mathrm{F}$ that $\lim_{m\to+\infty}\mathrm{F}(h_m) = \mathrm{F}(f)$.
% \begin{equation}\label{Eq: prop: upper semicont continuity measure-theoretic}
% \lim_{m\to+\infty}\mathrm{F}(h_m) = \mathrm{F}(f).
% \end{equation}
Furthermore, note that, due to the non-increasing character of $(g_{n,m})_{m\in\nats{}}$ and $(f_n)_{n\in\nats{}}$,
\begin{align*}
h_m(\omega) &= \min\{g_{n,m}(\omega) \colon 0\leq n \leq m\} \\
&\geq \min\{f_{n}(\omega) \colon 0\leq n \leq m\} = f_m(\omega),
\end{align*}
for all $m\in\nats{}$ and all $\omega\in\samplespace{}$.
So, $f_m \leq h_m$ for all $m\in\nats{}$, which by the monotonicity of $\mathrm{F}$ implies that 
\begin{equation*}
\lim_{m\to+\infty}\mathrm{F}(f_m)
\leq \lim_{m\to+\infty}\mathrm{F}(h_m) = \mathrm{F}(f).
\end{equation*}
The converse inequality---that $\lim_{m\to+\infty}\mathrm{F}(f_m)\geq\mathrm{F}(f)$---follows from the non-increasing character of $(f_n)_{n\in\nats{}}$ and the monotonicity of $\mathrm{F}$.

Finally, that the complementary statement holds for any $\mathrm{F}\colon\extvariables{}\to\extreals{}$ that is (monotone and) continuous with respect to non-decreasing sequences of finitary gambles, can easily be deduced from what we have just proved above, and the fact that $f\in\extvariables{}$ is l.s.c. if and only if $-f$ is an u.s.c. variable.
Indeed, the operator $\mathrm{F}'$ defined by $\mathrm{F}'(f)\coloneqq-\mathrm{F}(-f)$ for all $f\in\extvariables{}$ satisfies monotonicity and continuity with respect to non-increasing sequences of finitary gambles, so it follows that $\mathrm{F}'$ is also continuous with respect to non-increasing sequences of u.s.c. variables that are bounded above.
As a result, $\mathrm{F}$ is continuous with respect to non-decreasing sequences of l.s.c. variables that are bounded below.
\end{proof}

% \section{Proofs of the Results in Section~\ref{Sect: equivalence for measurable variables}}\label{Sect: Appendix C}

\begin{proofof}[Lemma~\ref{lemma: real upper semicont functions are bounded above}]
Recall from Lemma~\ref{lemma: semicontinuous functions are monotone limits of n-measurables} that $f$ is the pointwise limit of a non-increasing sequence $(f_n)_{n\in\nats{}}$ of finitary (bounded below) variables.
Assume \emph{ex absurdo} that $f$ is not bounded above.
Then, for each $n\in\nats{}$, since $f_n \geq \inf_{m\in\nats{}} f_m = f$, it follows that $f_n$ is also not bounded above.
Since each $f_n$ can only take a finite number of different values---because it is finitary and $\statespace{}$ is finite---we must have that $f_n(\omega) = +\infty$ for at least one $\omega\in\samplespace{}$.
So, for each $n\in\nats{}$, the set $A_n\coloneqq\{\omega\in\samplespace{}\colon f_n(\omega) = +\infty\}$ is non-empty.
Moreover, since $(f_n)_{n\in\nats{}}$ is non-increasing, $(A_n)_{n\in\nats{}}$ is also non-increasing and therefore, $\cap_{i=1}^n A_i = A_n \not=\emptyset$ for all $n\in\nats{}$.
So $(A_n)_{n\in\nats{}}$ has the finite intersection property.
Then, since $\samplespace{}$ is compact, it follows from \cite[Theorem~17.4]{willard2004general} that the sets $(A_n)_{n\in\nats{}}$ have a non-empty intersection if each of the $A_n$ is closed.
We proceed to show that the sets $A_n$ are closed.
Note that because each $f_n$ is finitary, the set $A_n$ is a finite union of cylinder events.
In particular, there is some $k\in\nats$ and some $S_n\subseteq\statespace{}^k$ such that $A_n = \cup_{x_{1:k}\in S_n} \Gamma(x_{1:k})$.
Since $\cup_{x_{1:k}\in \statespace{}^k} \Gamma(x_{1:k}) = \samplespace{}$, this implies that $A_n^c = \cup_{x_{1:k}\in \statespace{}^k\setminus S_n} \Gamma(x_{1:k})$ is a finite union of cylinder events and therefore, by Lemma~\ref{lemma: metric topology is product topology}, it is open.
So $A_n$ is closed
% ; it is the finite union of the cylinder events $\Gamma(x_{1:n})$ where $x_{1:n}\in\statespace{}^n$ and for which $f_n(x_{1:n}) = +\infty$.
% So the complement $A_n^c$ is a finite intersection of complements $\Gamma(x_{1:n})^c$ where $x_{1:n}\in\statespace{}^n$ and for which $f_n(x_{1:n}) = +\infty$.
% Any such complement $\Gamma(x_{1:n})^c$ is the finite union $\cup_{y_{1:n}\not=x_{1:n}} \Gamma(y_{1:n})$ of cylinder events, and therefore, by Lemma~\ref{lemma: metric topology is product topology}, it is open.
% Hence, since $A_n^c$ is a finite intersection of complements of the form $\Gamma(x_{1:n})^c$, it is open too.
% This in turn implies that $A_n$ is closed.
and we can therefore apply \cite[Theorem~17.4]{willard2004general} to find that $\cap_{n\in\nats} A_n \not= \emptyset$.
Then, for any $\omega\in\cap_{n\in\nats} A_n$, since $\omega\in A_n$ for all $n\in\nats{}$, it follows from the definition of the sets $A_n$ that $f_n(\omega)=+\infty$ for all $n\in\nats{}$.
As a consequence, $f(\omega)=\lim_{n\to+\infty}f_n(\omega) = +\infty$, which is in contradiction with the fact that $f$ is real-valued.
\end{proofof}

\begin{lemma}\label{lemma: global properties}
Consider any $\probtree{}$, any $\smash{\overline{\mathrm{Q}}}$ and any $s\in\situations{}$.
Then, for all $f,g\in\extvariables{}$ and $\mu\in\reals{}$, we have that
\begin{enumerate}[leftmargin=*,ref={\upshape{}E\arabic*},label={\upshape{}E\arabic*}.,itemsep=3pt]
	\item \label{global monotonicity}
$f\leq g \Rightarrow \mupprev{\probtree{}}(f \vert s) \leq \mupprev{\probtree{}}(g \vert s)$ \hfill \textnormal{[monotonicity];}
\item \label{global constant additivity}
$\mupprev{\probtree{}}(f + \mu \,\vert s) \leq \mupprev{\probtree{}}(f \vert s) + \mu$ \hfill \textnormal{[constant additivity],}
\end{enumerate}
and similarly for\/ $\upprev{}_{\smash{\overline{\mathrm{Q}}}}$.
\end{lemma}
\begin{proof}
To prove both properties for $\mupprev{\probtree{}}$, consider any compatible $p\sim\probtree{}$.
That Properties~\ref{global monotonicity} and~\ref{global constant additivity} hold for $\mupprev{p}$ follows from the fact that they are satisfied by the expectation $\smash{\prev{}_{p}}$ (if it exists; see \cite[Properties~M1 and~M2]{TJOENS202130}) together with Equation~\eqref{Eq: upper integral}.
Since $\mupprev{\probtree{}}$ is then simply the upper envelope of all $\smash{\mupprev{p}}$ with $\smash{p\sim\probtree{}}$, it follows that both properties are also satisfied by $\smash{\mupprev{\probtree{}}}$.
Furthermore, that $\upprev{}_{\smash{\overline{\mathrm{Q}}}}$ satisfies monotonicity is immediate from \ref{P monotonicity} and Theorems~\ref{theorem: Vovk is the largest} and~\ref{theorem: vovk and axiomatic are equal}.
That it is also constant additive follows from \cite[Proposition~7 (V4)]{TJOENS202130} and, again, Theorems~\ref{theorem: Vovk is the largest} and~\ref{theorem: vovk and axiomatic are equal}.
\end{proof}

\begin{proposition}\label{Prop: mupprev is capacity}
For any $s\in\situations{}$, the restrictions of \/ $\mupprev{\probtree{}}(\cdot \vert s)$ and \/ $\upprev{}_{\smash{\overline{\mathrm{Q}}}}(\cdot \vert s)$ to $\nnegextvariables$ are \/ $\samplespace{}$-capacities.
\end{proposition}
\begin{proof}
% [Proposition~\ref{Prop: mupprev is capacity}]
Property~\ref{capacity: monotonicity} follows for both $\mupprev{\probtree{}}(\cdot\vert s)$ and $\upprev{}_{\smash{\overline{\mathrm{Q}}}}(\cdot\vert s)$ from Lemma~\ref{lemma: global properties} (\ref{global monotonicity}) above.
That $\mupprev{\probtree{}}(\cdot\vert s)$ and $\upprev{}_{\smash{\overline{\mathrm{Q}}}}(\cdot\vert s)$ satisfy Property~\ref{capacity: upward continuity} follows from Proposition~\ref{prop: upward continuity measure-theoretic} and the fact that $\nnegextvariables{}\subseteq\bextvariables{}$.
% It is also satisfied by $\upprev{}_{\smash{\overline{\mathrm{Q}}}}(\cdot\vert s)$ because of \cite[Theorem~9(i)]{TJOENS202130} and, again, the fact that $\nnegextvariables{}\subseteq\bextvariables{}$.
Finally, that they satisfy Property~\ref{capacity: downward continuity} follows from Proposition~\ref{prop: upper semicont continuity}, together with the fact that, as a consequence of Lemma~\ref{lemma: real upper semicont functions are bounded above}, u.s.c. variables in $\nnegvariables{}$ are always bounded above.
% In a similar way, we can use Proposition~\ref{prop: upper semicont continuity} and Lemma~\ref{lemma: real upper semicont functions are bounded above} to establish that Property~\ref{capacity: downward continuity} holds for $\upprev{}_{\smash{\overline{\mathrm{Q}}}}(\cdot\vert s)$.
\end{proof}

\begin{proofof}[Theorem~\ref{theorem: equivalence for measurable functions}]
Let $f\in\bextvariables{}$ be bounded below and $\mathscr{F}$-measurable.
Since $f$ is bounded below, and both $\mupprev{\probtree{}}$ and $\upprev{}_{\smash{\overline{\mathrm{Q}}}}$ are constant additive (see Lemma~\ref{lemma: global properties} (\ref{global constant additivity})), we may assume without loss of generality that $f$ is non-negative---and therefore, that $f\in\nnegextvariables{}$. 
Then, according to Theorem~\ref{Theorem: Choquet}, the variable $f$ is universally capacitable.
Since $\mupprev{\probtree{}}(\cdot\vert s)$ and $\upprev{}_{\smash{\overline{\mathrm{Q}}}}(\cdot\vert s)$ are both $\samplespace{}$-capacities by Proposition~\ref{Prop: mupprev is capacity}, this implies that 
\begin{align*}
\mupprev{\probtree{}}(f \vert s)&=\sup\Bigl\{\mupprev{\probtree{}}(g \vert s)\colon g\in\nnegvariables{}\text{, $g$ is u.s.c.} \text{ and } f\geq g \Bigr\}
\end{align*}
and
\begin{align*}
\upprev{}_{\smash{\overline{\mathrm{Q}}}}(f \vert s)&=\sup\Bigl\{\upprev{}_{\smash{\overline{\mathrm{Q}}}}(g \vert s)\colon g\in\nnegvariables{}\text{, $g$ is u.s.c.} \text{ and } f\geq g \Bigr\}.
\end{align*}
Now recall Corollary~\ref{corollary: equivalence for upper semicontinuous functions}, which says that $\mupprev{\probtree{}}(h \vert s)=\upprev{}_{\smash{\overline{\mathrm{Q}}}}(h \vert s)$ for all u.s.c. variables $h\in\extvariables{}$ that are bounded above.
Since all u.s.c. variables $g\in\nnegvariables{}$ are automatically bounded above due to Lemma~\ref{lemma: real upper semicont functions are bounded above}, we obtain that $\mupprev{\probtree{}}(f \vert s)=\upprev{}_{\smash{\overline{\mathrm{Q}}}}(f \vert s)$.
\end{proofof}

\begin{proofof}[Corollary~\ref{corollary: equivalence}]
The statement for the upper expectations follows immediately from Theorem~\ref{theorem: equivalence for non-decreasing limits of n-measurables} and Theorem~\ref{theorem: equivalence for measurable functions}.
To prove the statement for the lower expectations, we distinguish two cases.
If $f$ is the pointwise limit of a monotone sequence of finitary gambles, then the same holds for $-f$, and hence, by Theorem~\ref{theorem: equivalence for non-decreasing limits of n-measurables}, $\mupprev{\probtree{}}(-f \vert s) = \upprev{}_{\smash{\overline{\mathrm{Q}}}}(-f \vert s)$.
This implies by conjugacy that $-\mlowprev{\probtree{}}(f \vert s) = -\lowprev{}_{\smash{\overline{\mathrm{Q}}}}(f \vert s)$ and therefore that $\mlowprev{\probtree{}}(f \vert s) = \lowprev{}_{\smash{\overline{\mathrm{Q}}}}(f \vert s)$.
On the other hand, if $f$ is an $\mathscr{F}$-measurable gamble, then so is $-f$, and therefore, by Theorem~\ref{theorem: equivalence for measurable functions}, we have that $\mupprev{\probtree{}}(-f \vert s) = \upprev{}_{\smash{\overline{\mathrm{Q}}}}(-f \vert s)$.
Conjugacy then again implies that $\mlowprev{\probtree{}}(f \vert s) = \lowprev{}_{\smash{\overline{\mathrm{Q}}}}(f \vert s)$.
\end{proofof}
}{} %end the arxiv-toggle
\end{document}